\documentclass[final]{siamart190516}

\usepackage[draft]{prelim2e} \usepackage{scrtime}

\usepackage[ngerman,english]{babel} 

\usepackage{todonotes}

\usepackage{amssymb,amsmath,amsfonts} 
\usepackage{mathrsfs}
\usepackage{nicefrac}
\usepackage{graphicx}
\usepackage{caption}
\usepackage{subcaption}
\usepackage{color}
\usepackage{algorithm,algpseudocode}

\newtheorem{assumption}[theorem]{Assumption}

\newtheorem{remark}[theorem]{Remark}

\newcommand{\TimesFont}{\RequirePackage{times}\RequirePackage[scaled=0.92]{helvet}}
\TimesFont

\definecolor{jz}{rgb}{0.1,0.45,0.1}

\definecolor{aw}{rgb}{0,0,1}

\definecolor{sw}{rgb}{1,0.53,0.0}

\newcommand{\cB}{\mathcal{B}}

\newcommand{\cD}{\mathcal{D}}

\newcommand{\cN}{\mathcal{N}}
\newcommand{\cO}{\mathcal{O}}

\newcommand{\cX}{\mathcal{X}}

\newcommand{\N}{\mathbb{N}}

\newcommand{\R}{\mathbb{R}}

\DeclareMathOperator{\Tr}{Tr}
\DeclareMathOperator{\spn}{span}

\usepackage{mathtools}

\usepackage{color}

\definecolor{darkred}{RGB}{139,0,0}
\definecolor{darkgreen}{RGB}{0,100,0}
\definecolor{darkmagenta}{RGB}{139,0,139}
\definecolor{darkpurple}{RGB}{110,0,180}
\definecolor{darkblue}{RGB}{40,0,200}
\definecolor{darkorange}{RGB}{255,140,0}
\begin{document}\selectlanguage{english}


\title{Gradient flow structure and convergence analysis of the ensemble Kalman inversion for nonlinear forward models}

\author{Simon Weissmann\thanks{Universit\"at Heidelberg, Interdisziplin\"ares Zentrum f\"ur Wissenschaftliches Rechnen, D-69120 Heidelberg, Germany ({\tt simon.weissmann@uni-heidelberg.de})}
}

\date{}

\maketitle

\begin{abstract} 
The ensemble Kalman inversion (EKI) is a particle based method which has been introduced as the application of the ensemble Kalman filter to inverse problems.  In practice it has been widely used as derivative-free optimization method in order to estimate unknown parameters from noisy measurement data.  For linear forward models the EKI can be viewed as gradient flow preconditioned by a certain sample covariance matrix.  Through the preconditioning the resulting scheme remains in a finite dimensional subspace of the original high-dimensional (or even infinite dimensional) parameter space and can be viewed as optimizer restricted to this subspace.  For general nonlinear forward models the resulting EKI flow can only be viewed as gradient flow in approximation.  In this paper we discuss the effect of applying a sample covariance as preconditioning matrix and quantify the gradient flow structure of the EKI by 
controlling the approximation error through the spread in the particle system.  The ensemble collapse on the one side leads to an accurate gradient approximation, but on the other side to degeneration in the preconditioning sample covariance matrix. In order to ensure convergence as optimization method we derive lower as well as upper bounds on the ensemble collapse. Furthermore, we introduce covariance inflation without breaking the subspace property intending to reduce the 
collapse rate of the ensemble such that the convergence rate improves. In a numerical experiment we apply EKI to a nonlinear elliptic boundary-value problem and illustrate the dependence of EKI as derivative-free optimizer on the choice of the initial ensemble.


\end{abstract}

\noindent {\footnotesize
\\
	\noindent {\bf AMS(MOS) subject classifications.} 37C10, 65N21, 65N75, 93D05
}

\tableofcontents


\section{Introduction} In inverse problems we are concerned with the task of recovering some unknown quantity of interest, which can not be observed directly.  Its research area has a wide range of applications in science and engineering,  including among others medical imaging and geophysics.  Since inverse problems are typically ill-posed,  much research focuses on theoretical and practical analysis of regularization methods.  Regularization techniques help to overcome instability issues arising for example through noise in the data.  Beside deterministic regularization methods, the Bayesian approach formulates inverse problems in a statistical framework.  By formulating a probabilistic approach, it incorporates uncertainty of the underlying model. The regularization can then be viewed as incorporating prior information through a probability distribution on the unknown parameter.  The solution of the Bayesian inverse problem is given by the posterior distribution - the distribution of the unknown parameter conditioned on the realization of the observation. The resulting posterior distribution is often not accessible directly, such that sampling or suitable integration methods are needed.

In this document, we focus on a particle based method which is commonly used for data assimilation problems - the ensemble Kalman filter (EnKF).  The EnKF has been introduced by Evensen \cite{GE03,GE09} and more recently, formulated for solving inverse problems \cite{ILS13}. The application of the EnKF to inverse problem has been established as widely used tool and is known as ensemble Kalman inversion (EKI).  
The aim of this manuscript is to analyse EKI as derivative-free optimization method for general nonlinear forward models. 
Our analysis will be based on the continuous-time formulation of the scheme, which can be viewed as (stochastic) ordinary differential equation ((S)ODE) in time, with focus on the deterministic setting, i.e.~without perturbed observations.  We will quantify the approximation of derivatives through the ensemble spread and derive sufficient lower and upper bounds on the collapse in order to verify convergence of the scheme.

\subsection{Literature review}
The EnKF has been introduced by Evensen \cite{GE03} as particle based approximation of the filtering distribution arising in data assimilation problems.  The method has been applied in the context of Bayesian inverse problems \cite{chen2012ensemble,emerick2013ensemble}. An important focus lies in the analysis of the large ensemble size limit, which has been studied  for linear and Gaussian models \cite{L2009LargeSA, doi:10.1137/140965363} as well as in a nonlinear setting \cite{doi:10.1137/140984415} and for the ensemble square root filter \cite{LS2021_c}.  Viewing the particle system as Monte Carlo method had led to much focus on the formulation of multilevel variants \cite{doi:10.1137/15M100955X,2016arXiv160808558C,HST2020,chada2021multilevel}.  Moreover,  ensemble Kalman methods have been studied in the long time behavior \cite{0951-7715-27-10-2579,47c5c78ebb8c44ef9d8d5e0d23e23c13,0951-7715-29-2-657}.
The accuracy for a fixed ensemble size has been studied in \cite{Tong2018,doi:10.1002/cpa.21722} for the EnKF and in \cite{delmoral2018,doi:10.1137/17M1119056} for the ensemble Kalman-Bucy filter. Furthermore, a high focus of research for ensemble Kalman methods lies in the continuous-time formulation \cite{bergemann2010localization,bergemann2010mollified,Reich2011, LS2021,LS2021_b,lange2021derivation}.

In \cite{ILS13} the authors propose the application of the EnKF to inverse problems.  In the literature EKI has been analysed as sequential Monte Carlo type methods as well as derivative-free optimization method. In the setting of linear forward maps and Gaussian prior assumptions it is well-known that EKI approximates the posterior distribution in its mean field limit. However, in \cite{ErnstEtAl2015} it has been demonstrated that in general nonlinear settings the resulting estimator is not consistent with respect to the posterior distribution. In the mean field limit EKI has been analysed based on the Fokker--Planck equation \cite{DL2021_b,MHGV2018}. 
The continuous-time limit of EKI has been formally derived in \cite{SS17} and theoretically analysed in \cite{BSW18,BSWW2021}. In the continuous-time formulation EKI can be viewed as derivative-free optimization method \cite{SS17} and has shown promising results for the training task in different machine learning applications \cite{KS18,GSW2020}.  Building up on the continuous-time formulation there has been much analysis on deterministic EKI, which ignores the diffusion of the underlying SDE, and stochastic EKI including the perturbed observations.  For linear models well-posedness and convergence results have been derived for the deterministic formulation \cite{SS17} and the stochastic formulation \cite{BSWW19}.  The nonlinear setting has been analysed in  \cite{Chada2019ConvergenceAO} where the authors consider the discrete time setting with a non constant step size scheme. The authors include variance inflation breaking the subspace property in order to verify convergence of the method.  Moreover, EKI has been studied for nonlinear models in the mean field limit \cite{DLL2020}, where weights have been incorporated in order to correct the resulting posterior estimate. In our presented analysis it turns out that the gradient flow structure highly depends on the behavior of the sample covariance which is used as preconditioner.  The dynamical behavior of EKI and its sample covariance has been described and analysed  by a spectral decomposition \cite{bungert2021}.  Furthermore, in \cite{tong2022localization} a localization of the sample covariance has been introduced and analysed based on the deterministic continuous-time formulation for nonlinear forward models.  Since EKI is applied as solver for inverse problems, one has to handle noise in the data. In this context, an early stopping criterion has been proposed in \cite{SS17b} and discrete regularization has been analysed in \cite{Iglesias2015,2016InvPr..32b5002I}. The focus in this document is on the Tikhonov regularized modification of EKI (TEKI) which has been introduced in \cite{CST19} and further analysed in \cite{weissmann2021adaptive}.  Further, adaptive regularization methods within EKI have been studied in \cite{parzer2021convergence,iglesias2020adaptive}.

From an alternative point of view, EKI has been modified to a particle based sampling method \cite{AGFHWLAS2019,GNR2020,DL2021,RW2021}.  Shifting the noise in the observations to the parameter space itself the resulting SDE can be treated as Langevin dynamic, where the ergodicity can be used to build a sampling method.

\subsection{Main contributions}

In this article we are going to quantify EKI as derivative-free optimization method in the context of general nonlinear forward maps. We employ a gradient flow structure and present a convergence analysis based on Lyapunov functions.  With this document, we extend the convergence analysis for TEKI presented in \cite{CST19} where the convergence result was based on linear forward models.  We make the following contributions:
\begin{itemize}
\item We present the well-posedness of the TEKI flow by proving unique existence of solutions for \eqref{eq:cont_deterministicTEKI}. Furthermore, we view the TEKI flow as approximate gradient flow by decomposing the flow into the preconditioned gradient direction and the approximation error resulting through Taylor expansion. 
\item We quantify the approximation error through the spread in the particle system. While the ensemble collapses in time with rate $1/t$, we will see that the approximation error degenerates with rate $(1/t)^{3/2}$. In contrast, we prove that the sample covariance remains strictly-positive definite as operator acting on the subspace spanned by the initial ensemble spread {$\mathcal B = u_0^\perp + \spn\{u_0^{(j)}-\bar u_0\mid j=1,\dots,J\}$ for a specific $u_0^\perp\in\mathcal X$}.
\item We describe the effect of preconditioning through the sample covariance. Therefore, we view the minimization task of $\Phi_R$ over the subspace {$\mathcal B$} as equality constrained optimization method and derive a Polyak-Łojasiewicz (PL)-type inequality restricted to the subspace {$\mathcal B$}. Under strong convexity and smoothness of $\Phi_R$ we are then able to prove convergence with rate $(1/t)^{\frac1\alpha}$, $\alpha\ge1$.
\item We incorporate covariance inflation without breaking the subspace property. The inflation sufficiently slows down the ensemble collapse such that the convergence rate $\alpha$ can be improved to $(1-\rho)\alpha$, where $\rho\in[0,1)$ is the inflation factor.  Furthermore, we view our proposed covariance inflation scheme as generalization of the ensemble square root filter (ESRF) applied to inverse problems.
\end{itemize}

\paragraph{Outline} The remainder of this manuscript is structured as follows. In Section~\ref{sec:mathsetup} we provide the mathematical background for (T)EKI, and we formulate our main convergence result with the corresponding assumptions in Section~\ref{sec:main}.  We provide a row of auxiliary properties of the TEKI flow in Section~\ref{sec:prelim} which are then applied to present the proof of our main Theorem~\ref{thm:main} in Section~\ref{sec:proof_main}.  The incorporation of the covariance inflation and the improved convergence analysis are presented in Section~\ref{sec:ESRF}.  In Section~\ref{sec:numerics} we illustrate the presented convergence results in a one-dimensional elliptic boundary-value problem. The document will be closed with a summary and outlook for open research directions in Section~\ref{sec:conclusion}.

\section{Mathematical setup}\label{sec:mathsetup} 
In the following we will introduce the mathematical background. We consider the inverse problem of recovering $u\in\mathcal X$ given noisy observations
\begin{equation}\label{eq:IP}
y = G(u) + \eta,
\end{equation}
where {$\mathcal X$ is some Hilbert space}, $G:\mathcal X\to\mathbb R^K$ is the possibly nonlinear forward map and $\eta$ denotes additive Gaussian noise, i.e.~$\eta\sim\mathcal N(0,\Gamma)$ for some symmetric positive definite noise covariance matrix $\Gamma\in\mathbb R^{K\times K}$.  Due to ill-posedness of inverse problems, we typically consider the task of minimization of a regularized loss functional of the form

\begin{equation*}
\min_{x\in\mathcal X}\ \mathcal L_{\mathbb R^K}(G(x),y) + \mathcal R_{\mathcal X}(x),
\end{equation*}
where we describe the discrepancy to the data through $\mathcal L_{\mathbb R^K}:\mathbb R^K\times\mathbb R^K \to \mathbb R_+$ and incorporate prior information through the regularization function $\mathcal R_\mathcal X:\mathcal X\to\mathbb R_+$.  We refer to 
\cite{EHN1996,benning_burger_2018} for an overview of various types of regularization methods.  Classical choices of regularization include Tikhonov regularization \cite{EKN1989} and total variation regularization \cite{Chambolle2009AnIT,ROF1992}.  Throughout this document, our focus will be on the finite dimensional case $\mathcal X=\mathbb R^{n_x}$ and the Tikhonov regularized lossfunction of the form
\begin{equation}\label{eq:regularized_lossfunction}
\Phi_R(x) = \Phi(x;y) + R(x), \quad R(x) = \frac12\|x\|_{C_0}^2,
\end{equation}
where we define the least-square data misfit by
\begin{equation}\label{eq:datamisfit}
\Phi(x;y) = \frac12 \|G(x)-y\|_\Gamma^2
\end{equation}
and $R$ denotes the regularization function with symmetric positive definite regularization matrix $C_0\in\R^{n_x\times n_x}$. Here,  we have introduced the notation $\|\cdot\|_{\Gamma}:=\|\Gamma^{-1/2}\cdot\|$, $\|\cdot\|_{C_0}=\|C_0^{-1/2}\cdot\|$, where $\|\cdot\|$ denotes the euclidean norm in $\R^K$ and $\R^{n_x}$ respectively.  Moreover, we will make use of $\|\cdot\|_{\mathrm{HS}}$ as notation for the Hilbert-Schmidt operator norm. 

The Tikhonov regularization is closely connected to the Bayesian approach for inverse problems \cite{AMS10} which incorporates regularization from a statistical point of view.  In a probabilistic model we let $u$ be an $\mathcal X$-valued random variable with prior distribution $\mu_0$ stochastically independent of the noise $\eta$.  When viewing $(u,y)$ as a jointly distributed random variable on $\R^K\times\mathcal X$,  the solution of the Bayesian inverse problem is given by the posterior distribution of $u\mid y$ 
\begin{equation}\label{eq:posterior}
\mu(\mathrm{d}x) = \frac1Z\exp(-\Phi(x;y))\mu_0({\mathrm{d}}x),
\end{equation}
with normalization constant 
\[ 
Z:= \int_{\R^{n_x}}\exp(-\Phi(x;y))\mu_0({\mathrm{d}}x).
\]
We note for a Gaussian prior assumption $\mu_0 = \cN(0, C_0)$ the maximum a-posteriori estimate computes as
\[
\min_{x\in\mathcal X}\ \Phi(x;y) + \frac12\|x\|_{C_0}^2,
\]
which relates the Bayesian approach for inverse problems to the Tikhonov regularization through \eqref{eq:regularized_lossfunction}. Throughout this manuscript we assume that $y\in\R^{n_y}$ is fixed and suppress the dependence of $\Phi$ on $y$ by writing $\Phi(x) = \Phi(x;y)$.

\subsection{The ensemble Kalman filter applied to inverse problems}

The EKI has been introduced as the application of the EnKF to inverse problems \cite{ILS13}. We will follow the derivation of EKI as sequential Monte Carlo method aiming to approximate the posterior distribution. Therefore, we introduce
the tempered distribution
\begin{equation}\label{eq:tempering}
\mu_{n+1}({\mathrm{d}}u) = \frac1{Z_n} \exp(-h\Phi(u;y))\mu_n({\mathrm{d}}u), {\quad n=1,\dots,N,\quad N\in\N,}
\end{equation}
with $h=1/N$ and normalizing constants $Z_n$, where $\mu_0$ is the prior distribution and $\mu_N$ corresponds to the posterior distribution.  {This step incorporates an artificial discrete time system shifting weight from prior to posterior distribution.} Given a sample from the prior distribution, the EKI evolves the sample as particle system through Gaussian approximation steps. Let $(u_0^{(j)})_{j\in\{1,\dots,J\}}$ be the initial ensemble of size $J\ge 2$ and consider the empirical approximation of the tempering distribution defined in \eqref{eq:tempering}
\[
\mu_n(\mathrm{d}u)\approx \frac1J\sum_{j=1}^J\delta_{u_n^{(j)}}(\mathrm{d}u). 
\]
Defining the following empirical means and covariances
\begin{align*}
C(u)&=\frac1J\sum_{j=1}^J(u^{(j)}-\bar u)\otimes (u^{(j)}-\bar u),\quad \bar u = \frac1J\sum_{j=1}^J u^{(j)},\\
 C^{G,G}(u)&=\frac1J\sum_{j=1}^J(G(u^{(j)})-\bar G)\otimes (G(u^{(j)})-\bar G),\quad \bar G = \frac1J\sum_{j=1}^J G(u^{(j)}),\\
C^{u,G}(u) &= \frac1J\sum_{j=1}^J(u^{(j)}-\bar u)\otimes (G(u^{(j)})-\bar G), \\ C^{G,u}(u) &= \frac1J\sum_{j=1}^J(G(u^{(j)})-\bar G)\otimes (u^{(j)}-\bar u),
\end{align*}
the ensemble Kalman iteration for the current particle system $(u_n^{(j)})_{j\in\{1,\dots,J\}}$ is given by
\begin{equation}\label{eq:discreteEKI}
 u_{n+1}^{(j)} = u_n^{(j)} + C^{u,G}(u_n)(C^{G,G}(u_n) + h^{-1}\Gamma)^{-1} (y_{n+1}^{(j)} - G(u_n^{(j)})),\quad j=,1\dots,J,
\end{equation}
where $h>0$ is a artificial step size and $y_{n+1}^{(j)}$ are perturbed observations
\[y_{n+1}^{(j)} = y + \xi_{n+1}^{(j)},\quad \xi_{n+1}^{(j)} \sim \mathcal N(0,h^{-1}\Gamma)\]
with $\xi_{n+1}^{(j)}$ being i.i.d.~with respect to both $j$ and $n$.  Considering the limit $h\to0$ the discrete EKI \eqref{eq:discreteEKI} can be viewed as time discretization of the system of coupled SDEs
\begin{equation}\label{eq:cont_stochasticEKI}
{\mathrm d}u_t^{(j)} = C^{u,G}(u_t)\Gamma^{-1}(y-G(u_t^{(j)}))\, {\mathrm d}t + C^{u,G}(u_t)\Gamma^{-1/2}\,{\mathrm d}W_t^{(j)},
\end{equation}
where $W^{(j)}$ are independent Brownian motions on $\mathbb R^K$. The continuous-time limit has been formulated in \cite{SS17} and theoretical verified in \cite{BSW18, BSWW2021}. 

Since the EKI is not consistent with the posterior distribution for general nonlinear forward maps \cite{ErnstEtAl2015}, we will focus on EKI as derivative-free optimization method.  Instead of considering perturbed observations, sometimes also referred to as stochastic EKI, the EKI is often analysed in its deterministic formulation represented by a system of coupled ODEs
\begin{equation}\label{eq:cont_deterministicEKI}
\frac{{\mathrm d}u_t^{(j)}}{{\mathrm d}t} = C^{u,G}(u_t)\Gamma^{-1}(y-G(u_t^{(j)})),
\end{equation}
which describes the drift term of the SDE formulation \cite{SS17,SS17b}. Considering a linear forward map $G(\cdot) = A\cdot$ with $A\in\mathcal L(\mathcal X,\mathbb R^K)$ the deterministic EKI can be viewed as preconditioned gradient flow written through
\begin{equation}\label{eq:cont_deterministicEKI2}
\frac{{\mathrm d}u_t^{(j)}}{{\mathrm d}t} = -C(u_t)\nabla\Phi(u_t^{(j)}).
\end{equation}
For general nonlinear forward maps the representation as preconditioned gradient flow holds only approximatively. In particular,  using Taylor approximations we can justify that
\[C^{u,G}(u) \approx C(u){\mathrm D}G^\ast(u),\]
where ${\mathrm D}G$ denotes the (Fréchet) derivative of $G$, and we are going to quantify the error between the nonlinear EKI and a preconditioned gradient flow. Since our aim is to view the EKI as optimization method for minimizing the data misfit functional $\Phi$, the resulting scheme is ill-posed due to the ill-posed inverse and it is natural to consider a regularized modification of the EKI. Therefore, the authors in \cite{CST19} proposed to incorporate Tikhonov regularization within EKI resulting in the deterministic continuous-time formulation as coupled system of ODEs
\begin{equation}\label{eq:cont_deterministicTEKI}
\frac{{\mathrm d}u_t^{(j)}}{{\mathrm d}t} = C^{u,G}(u_t)\Gamma^{-1}(y-G(u_t^{(j)}))-C(u_t)C_0^{-1}u_t^{(j)},
\end{equation}
which can then be viewed as derivative-free approximation to the preconditioned gradient flow
\begin{equation}\label{eq:cont_deterministicTEKI2}
\frac{{\mathrm d}u_t^{(j)}}{{\mathrm d}t} = -C(u_t)\nabla\Phi_R(u_t^{(j)}).
\end{equation}
We will see that this approximation is getting more accurate for a decreasing spread of the particle system. To do so, we are going to analyse the long-time behavior of \eqref{eq:cont_deterministicTEKI} for a fixed number of particles. Therefore, we will extend the convergence analysis presented in \cite{CST19} from the linear to the nonlinear setting. Here, it has been shown that assuming a linear forward model $G(\cdot) = A\cdot, $ $A\in\mathcal L(\mathcal X,\mathbb R^K)$ the resulting solution of the TEKI flow \eqref{eq:cont_deterministicTEKI} minimizes the objective function $\Phi_R$ in the subspace spanned by initial ensemble {spread $\mathcal B=u_0^\perp + \spn\{u_0^{(j)}-\bar u_0\mid j=1,\dots,J\}$ for a specific $u_0^\perp\in\mathcal X$}.

\section{Main result: Convergence of TEKI for nonlinear forward models}\label{sec:main}
In order to quantify the nonlinear EKI as derivative-free approximation we make the following assumptions.

\begin{assumption}\label{ass:smoothness}
We assume that the objective function $\Phi_R$ is $C^2(\cX;\R_+)$ and $\mu$-strongly convex, i.e.~
\[\Phi_R(x_1)-\Phi_R(x_2)\ge \langle \nabla\Phi_R(x_2),x_1-x_2\rangle + \frac{\mu}2\|x_1-x_2\|^2\]
for some $\mu>0$ and all $x_1,x_2\in\mathcal X$.  Moreover, we assume that $\Phi_R$ is $L$-smooth, i.e.
\[\Phi_R(x_1)-\Phi_R(x_2)\le \langle \nabla\Phi_R(x_2),x_1-x_2\rangle + \frac{L}2\|x_1-x_2\|^2\]
for some $L>0$ and all $x_1,x_2\in\mathcal X$ and therefore satisfies the PL-inequality
\[  \nu\|\nabla\Phi(x)\|^2\ge \Phi_R(x) - \Phi(x_\ast), \quad x\in\mathcal X\]
for some $\nu>0$, where $x_\ast$ is a stationary point, i.e.~$\nabla\Phi(x_\ast)=0$.
\end{assumption}

{In the above assumption, we suppose global $\mu$-strong convexity and $L$-smoothness, which is needed to prove convergence results as global optimizer.  In the case, where those assumptions are only satisfied locally, we can only expect the results to hold in local neighborhoods around critical points.  The necessary convergence analysis as local optimizer is work in progress.}

\begin{assumption}\label{ass:Taylor}
We assume that the forward map $G$ is $C^2(\cX;\R^{n_y})$, locally Lipschitz with constant $c_{\mathrm{lip}}>0$ and can be approximated linearly by
\begin{equation}\label{eq:linearization}
G(x_1) = G(x_2) + {\mathrm D}G(x_2)(x_1-x_2) + {\mathrm{Res}}(x_1,x_2),
\end{equation}
where the approximation error ${\mathrm{Res}}$ is bounded by
\begin{equation}\label{eq:Taylor_residual}
\|{\mathrm{Res}}(x_1,x_2)\|_{\R^{n_y}} \le b_{\mathrm{res}} \|x_1-x_2\|_{\cX}^2
\end{equation} 
for some $b_{\mathrm{res}}>0$ independent from $x_1,x_2\in\mathcal X$.
\end{assumption}

We note that Assumption~\ref{ass:Taylor} is satisfied for linear forward maps $G(\cdot) = A\cdot$, $A\in\mathcal L(\mathcal X,\mathbb R^K)$ with ${\mathrm{Res}}(x_1,x_2) = 0$ for all $x_1,x_2\in\mathcal X$. Furthermore, it is sufficient to assume that the Hessian of $G$ is uniformly bounded in the sense that $\langle z, {\mathrm D}^2 G(x)z\rangle \le c \|z\|^2$ for all $z\in\cX$. Alternatively, the tangential cone condition
\begin{equation*}
\|G(x_1)-G(x_2)-{\mathrm D}G(x_2)(x_1-x_2)\|\le \tilde b\|x_1-x_2\| \|G(x_1)-G(x_2)\| 
\end{equation*}
for some uniform $\tilde b>0$ together with Lipschitz continuity implies Assumption~\ref{ass:Taylor}. This condition is often assumed (locally) when studying iterative methods for nonlinear inverse problems \cite{S1995}.  {As soon as we are able to proof uniform boundedness of the solution of the TEKI flow \eqref{eq:cont_deterministicTEKI}, it is sufficient to consider Assumption~\ref{ass:Taylor} locally, i.e.~to assume that for $B>0$ there exists $b_{\mathrm{res}}(B)>0$ such that \eqref{eq:Taylor_residual} holds for all $x_1,x_2\in\mathcal X$ with $\|x_1\|_{\mathcal X},\|x_2\|_{\mathcal X}^2\le B$. While the proof of boundedness of the TEKI flow \eqref{eq:cont_deterministicTEKI} is typically challenging, one can enforce boundedness through a smooth modification of the forward map $G$, such that $G(x)=0$ for $\|x\|>B$, see for example \cite{CST19, Chada2019ConvergenceAO}. Considering this type of modification might be reasonable for practical implementation if it is known that solutions of the corresponding inverse problem should be bounded.}

Furthermore, we note that the above Assumption~\ref{ass:Taylor} can be weaken through the assumption $\|{\mathrm{Res}}(x_1,x_2)\| \le b_{\mathrm{res}} \|x_1-x_2\|^p$,  for $p> 2$.  However, to avoid technical details we will keep the assumption with $p=2$ for the main proof and state more details on the extension to $p>2$ in Remark~\ref{rem:higher_apprx} after the proof of our main result.

\begin{proposition}\label{prop:gradientflow_structure}
Suppose Assumption~\ref{ass:Taylor} is satisfied, then it holds true that
\begin{align*}
\frac{{\mathrm d}\frac1J\sum_{j=1}^J\left(\Phi_R(u_t^{(j)})\right)}{{\mathrm d}t} &\le -\frac1J\sum_{j=1}^J \langle \nabla \Phi_R(u_t^{(j)}), C(u_t)\nabla \Phi_R(u_t^{(j)})\rangle\\ &\quad+   b_1 J V_e(t)^{3/2}\frac{1}{J}\sum_{j=1}^J\left(\frac{1}{2}\|\nabla\Phi_R(u_t^{(j)})\|^2+ \Phi_R(u_t^{(j)}) \right),
\end{align*}
where $b_1>0$ is independent from $J$ and $V_e(t) :=\frac{1}{J}\sum_{k=1}^J \|u_t^{(k)}-\bar u_t\|^2$.
\end{proposition}

The above result quantifies the (approximate) gradient flow structure of TEKI. We will see, that $V_e(t)$ is of the order $1/t$, which means that the approximation error of the gradient through the covariance degenerates in time. However, the degeneration of $C(u_t)$ influences the preconditioning effect for the gradient as well. Although we can show, that the degeneration of $C(u_t)$ is bounded from below by the order $1/t$, it is not possible to imply lower bounds on the eigenvalues $C(u_t)$. This issue comes from the fact that $C(u_t)$ is a sample covariance and has at most rank $\min(n_x,J-1)$. To alleviate this problem, we will show that $C(u_t)$ will remain positive definite in the subspace spanned through the initial ensemble {spread}.  {Along the solution of the TEKI flow the particle system $(u_t^{(j)})_{j=1,\dots,J}$ remains in the subspace 
\[\cB := u_0^\perp + \spn\{u_0^{(j)} - \bar u_0 \mid j=1,\dots,J\} \subset \spn\{u_0^{j}\mid j=1,\dots,J\}=:\mathcal S,  \]
where $u_0^\perp = \bar u_0 - \mathcal P_E \bar u_0$, where $\mathcal P_E:= E(E^\top E)^{-1}E^T$ with $E=((u_0^{(1)}-\bar u_0)^\top,\dots,(u_0^{(J)}-\bar u_0)^\top)\in\mathbb R^{n_x\times J}$  is the orthogonal projection on the subspace $\spn\{u_0^{(j)} - \bar u_0 \mid j=1,\dots,J\}$.  {This subspace property also holds for general nonlinear forward maps $G$, see  Lemma~3.7 and Corollary~3.8 in \cite{CST19}.}}
Hence, our aim is to quantify TEKI as gradient flow w.r.t.~the constrained optimization problem

\begin{equation}\label{eq:constr_opti}
\min_{x\in {\mathcal B}}\ \Phi_R(x),
\end{equation}
where $(u_0^{(j)})_{j=1,\dots,J}$ is the initial ensemble for TEKI. In fact, a similar result has been shown in \cite{CST19} for the linear setting, where the gradient approximation is exact. We are going to extend the theory to nonlinear forward maps satisfying Assumptions~\ref{ass:smoothness} and \ref{ass:Taylor}.  Note that we will reformulate the optimization problem \eqref{eq:constr_opti} as equality constrained optimization problem, implying that there exists a unique minimizer $u_\ast\in\mathcal B$ of \eqref{eq:constr_opti}. In the following we formulate our main convergence result.

\begin{theorem}\label{thm:main}
Suppose Assumption~\ref{ass:smoothness} and Assumption~\ref{ass:Taylor} are satisfied, let $(u_t^{(j)},t\ge0)_{j=1,\dots,J}$ solve \eqref{eq:cont_deterministicTEKI} with linearly independent initial ensemble $(u_0^{(j)})_{j=1,\dots,J}$ and let $u_\ast\in\mathcal B$ be the unique global minimizer of \eqref{eq:constr_opti}. Then there exist $c_1,c_2>0$ such that
\begin{equation}\label{eq:main_result}
\frac{1}{J}\sum_{j=1}^J\Phi_R(u_t^{(j)})-\Phi_R(u_\ast) \le \left(\frac{c_1}{t+c_2}\right)^{\frac1\alpha},
\end{equation}
where $0<\alpha < \frac{L}{\mu}(\sigma_{\max}+c_{\mathrm{lip}}\lambda_{\max}\|C(u_0)\|_{\mathrm{HS}})$. 
\end{theorem}
{We note that the constants $c_1,c_2>0$ depend on various constants arising in Assumption~\ref{ass:smoothness}-\ref{ass:Taylor} and may be increasing for less well-behaved scenarios. However, these constants are independent of the time $t>0$ and the convergence still holds asymptotically.}

\section{Preliminaries}\label{sec:prelim}

In the following we present a row of auxiliary results, which are needed to prove the main convergence result in Theorem~\ref{thm:main}.

\subsection{Existence of solutions and ensemble collapse}

Recall that we assume a finite dimensional parameter space $\mathcal X = \mathbb R^{n_x}$. We start this section with verifying the well-posedness of the scheme. Therefore, we state that there exists a unique solution for the considered TEKI flow.
\begin{theorem}
Let $(u_0^{(j)})_{j=1,\dots,J}$ be the linearly independent initial ensemble and define its linear span $S = \spn\{u_0^{(j)}\mid j=1,\dots,J\}$. Then for all $T>0$ the equation \eqref{eq:cont_deterministicTEKI} has a unique solution $(u_t^{(j)}, t\in[0,T])\in C([0,T],\mathcal S)$.  In particular, it holds true that 
$\langle z, u_t^{(j)}\rangle = 0$
for all $z\in S^\perp$ and all $t\in[0,T]$, $j=1,\dots,J$.
\end{theorem}

\begin{proof}
Due to local Lipschitz continuity local solutions follow from standard ODE theory, see e.g.~\cite[Theorem~3.2]{CST19} for details. Moreover, we refer to \cite[Corollary~3.8]{CST19} for the subspace property, i.e.~while a solution exists it holds true that $\langle z, u_t^{(j)}\rangle = 0$
for all $z\in S^\perp$ and all $j=1,\dots,J$.  In the following, we extend the existence of solutions globally. Therefore, we consider the Lyapunov function
\begin{equation*}
V_{\bar u} (u_t) = \varphi(\|\bar u_t\|^2) + \frac{B_1}J\sum_{j=1}^J\|u_t^{(j)}-\bar u_t\|^2 + B_2
\end{equation*}
for $\varphi:\R\to\R$ with $\varphi(z) = \log(1+z)$ and sufficiently large $B_1,B_2>0$. We firstly observe that $z\varphi'(z) = \frac{z}{1+z}\le 1$ for $z\ge0$.  The time evolution of $\varphi(\|\bar u_t\|^2)$ is given by
\begin{align*}
 \frac{\mathrm{d}}{{\mathrm{d}}t}\varphi(\|\bar u_t\|^2) &= 2\varphi'(\|\bar u_t\|^2)\left(- \langle \bar u_t, C^{u,G}(u_t)\Gamma^{-1}(\bar G_t-y)\rangle - \langle \bar u_t, C(u_t)C_0^{-1}\bar u_t\rangle\right)\\
 &\le 2\varphi'(\|\bar u_t\|^2) \left(\|\bar u_t\|\|\bar G_t-y\|_\Gamma \|C^{u,G}(u_t)\Gamma^{-1/2}\|_{\mathrm{HS}} + \|\bar u_t\|\|\bar u_t\|_{C_0} \|C(u_t)C_0^{-1/2}\|_{\mathrm{HS}}\right)\\
&\le \varphi'(\|\bar u_t\|^2)^2\|\bar u_t\|^2 (2\|\bar G_t - G(\bar u_t)\|_\Gamma^2+2\|G(\bar u_t)-y\|_\Gamma^2) +  \|C^{u,G}(u_t)\Gamma^{-1/2}\|_{\mathrm{HS}}^2\\ &\quad+  \varphi'(\|\bar u_t\|^2)^2\|\bar u_t\|^2 \|\bar u_t\|_{C_0}^2 +  \|C(u_t)C_0^{-1/2}\|_{\mathrm{HS}}^2,
\end{align*}
where we have applied Young's inequality and 
\begin{equation*}
 \|\bar G_t-y\|_\Gamma^2\le (\|\bar G_t-G(\bar u_t)\|_\Gamma+\|G(\bar u_t)-y\|_\Gamma)^2\le 2\|\bar G_t-G(\bar u_t)\|_\Gamma^2+2\|G(\bar u_t)-y\|_\Gamma)^2
\end{equation*}
for the last inequality.
With Lipschitz continuity of $G$, Cauchy-Schwarz inequality and Jensen's inequality we have that 
\begin{align*}
\|\bar G - G(\bar u) \|^2 = \frac{1}{J^2}\sum_{j,l}^2 \langle G(u^{(j)})-G(\bar u),G(u^{(l)})-G(\bar u)\rangle &\le c_{\mathrm{lip}}^2\left(\frac{1}{J}\sum_{j=1}^J \|u^{(j)} - \bar u\|\right)^2 \\
&\le c_{\mathrm{lip}}^2\frac{1}{J}\sum_{j=1}^J \|u^{(j)} - \bar u\|^2.
\end{align*}
Hence,  using $z\varphi'(z)\le 1$ we obtain
\begin{align*}
 \frac{\mathrm{d}}{{\mathrm{d}}t}\varphi(\|\bar u_t\|^2) &\le c_1\frac{1}{J}\sum_{j=1}^J \|u_t^{(j)} - \bar u_t\|^2+ \tilde c_2 \varphi'(\|\bar u\|^2)\Phi_R(\bar u_t)\\ &\quad+  \|C^{u,G}(u_t)\Gamma^{-1/2}\|_{\mathrm{HS}}^2+ \|C(u_t)C_0^{-1/2}\|_{\mathrm{HS}}^2\\
&\le c_1\frac{1}{J}\sum_{j=1}^J \|u_t^{(j)} - \bar u_t\|^2+ c_2 +  \|C^{u,G}(u_t)\Gamma^{-1/2}\|_{\mathrm{HS}}^2+ \|C(u_t)C_0^{-1/2}\|_{\mathrm{HS}}^2
\end{align*}
for some constants $c_1,c_2>0$. Here we have used, that $\Phi_R(u)\le c\|u\|^2$ for local Lipschitz continuous $G$. On the other side, we have
\begin{align*}
 \frac{{\mathrm d}\frac1J\sum_{j=1}^J \|u_t^{(j)}-\bar u_t\|^2}{{\mathrm d}t} &= -\frac{2}{J^2} \sum_{j,k=1}^J \langle u_t^{(j)}-\bar u_t,u_t^{(k)}-\bar u_t\rangle \langle G(u_t^{(k)})-\bar G_t,G(u_t^{(j)})-\bar G_t\rangle_{\Gamma}\\
					&\quad - \frac{2}{J^2} \sum_{j,k=1}^J \langle u_t^{(j)}-\bar u_t,u_t^{(k)}-\bar u_t\rangle \langle u_t^{(k)}-\bar u_t,C_0^{-1}(u_t^{(j)}-\bar u_t)\rangle\\
					&= -2\|C^{u,G}(u_t)\Gamma^{-1/2}\|_{\mathrm{HS}}^2-2\|C(u_t)C_0^{-1/2}\|_{\mathrm{HS}}^2.
\end{align*}
It follows for $B_1>c_1$ and $B_2>c_2$ that
\[\frac{\mathrm{d}}{{\mathrm{d}}t}V_{\bar u}(u_t) \le c V_{\bar u}(u_t)\]
and by application of Gronwall's inequality for all $T>0$ the solution of \eqref{eq:cont_deterministicTEKI} remains bounded.
\end{proof}

{We note that a similar type of Lyapunov function $\varphi$ has been used in \cite{BSWW2021} in order to show finite logarithmic moments of the stochastic formulation. While it is not clear whether $L_2$ bounds hold for (T)EKI, one is at least able to derive bounds on logarithmic moments. }

In the following we define the spread of the particle system by
\begin{equation}
e_t^{(j)} = u_t^{(j)}- \bar u_t.
\end{equation}
{From Lemma~3.7 and Corollary~3.8 in \cite{CST19} it is known that the TEKI flow satisfies the following extended subspace property.
\begin{proposition}[Corollary~3.8, \cite{CST19}]\label{prop:subspace_prop}
Let $(u_0^{(j)})_{j=1,\dots,J}$ be the linearly independent initial ensemble and define the subspace
\[\cB = u_0^\perp + \spn\{e_0^{(j)} \mid j=1,\dots,J\},  \]
where $u_0^\perp = \bar u_0 - \mathcal P_E \bar u_0$,  where $\mathcal P_E= E(E^\top E)^{-1}E^T$ with $E=((e_0^{(1)})^\top,\dots,(e_0^{(J)})^\top)\in\mathbb R^{n_x\times J}$  is the orthogonal projection on the subspace $\spn\{e_0^{(j)} \mid j=1,\dots,J\}$. Then for $T\ge 0$ the unique solution $(u_t^{(j)}, t\in[0,T])$  of the TEKI flow \eqref{eq:cont_deterministicTEKI} remains in $\mathcal B$, i.e.~
$u_t^{(j)} \in \mathcal B$ for all $t\in[0,T]$ and all $j=1,\dots,J$.
\end{proposition}
}

On the one side the spread of the particle system describes the accuracy of the gradient approximation, which gets more accurate for a less spread particle system. On the other side, the spread of the particle system determines the preconditioning effect through the sample covariance. For this reason we aim for a collapse of the ensemble but not too fast. In our first result, we prove that the ensemble collapses with rate $1/t$. {Recall that $C_0$ is assumed to be a fixed symmetric positive definite matrix, and hence all of its eigenvalues are strictly positive.}
\begin{lemma}\label{lem:ensemble_collapse}
Let $(u_t^{(j)},t\ge0)_{j=1,\dots,J}$ be the solution of \eqref{eq:cont_deterministicEKI} initialized by a linearly independent ensemble $(u_0^{(j)})_{j=1,\dots,J}$. Then the mapping $V_e:\mathbb R_+\to \mathbb R_+$ with $t\mapsto \frac{1}{J}\sum_{j=1}^J\|e_t\|^2$
is bounded by
\[V_e(t) \le \frac{1}{\frac{2\sigma_{\min}}{J} t + V_e(0)^{-1}},\]
where $\sigma_{\min}$ denotes the smallest eigenvalue of $C_0^{-1}$.
\end{lemma}
\begin{proof} 
The evolution of $V_e(t)$ is given by
\begin{align*}
 \frac{{\mathrm d}\frac1J\sum_{j=1}^J \|u_t^{(j)}-\bar u_t\|^2}{{\mathrm d}t} &= -\frac{2}{J^2} \sum_{j,k=1}^J \langle u_t^{(j)}-\bar u_t,u_t^{(k)}-\bar u_t\rangle \langle G(u_t^{(k)})-\bar G_t,G(u_t^{(j)})-\bar G_t\rangle_{\Gamma}\\
					&\quad - \frac{2}{J^2} \sum_{j,k=1}^J \langle u_t^{(j)}-\bar u_t,u_t^{(k)}-\bar u_t\rangle \langle u_t^{(k)}-\bar u_t,C_0^{-1}(u_t^{(j)}-\bar u_t)\rangle\\
					&= -2\|C^{u,G}(u_t)\Gamma^{-1/2}\|_{\mathrm{HS}}^2\\ &\quad-\frac{2}{J^2} \sum_{j,k=1}^J \langle u_t^{(j)}-\bar u_t,u_t^{(k)}-\bar u_t\rangle \langle u_t^{(k)}-\bar u_t,C_0^{-1}(u_t^{(j)}-\bar u_t)\rangle\\
					&\le -\frac{2\sigma_{\min}}{J} \left(\frac1J\sum_{j=1}^J \|u_t^{(j)}-\bar u_t\|^2\right)^2.
\end{align*}
Hence, we have shown \[\frac{{\mathrm d}V_e(t)}{{\mathrm d}t} \le - \frac{2\sigma_{\min}}{J}V_e(t)^2\]
and the assertion follows by Lyapunov-type argument.
\end{proof}

In the following result, we state that, although the sample covariance degenerates along the particle evolution, it remains a strictly positive definite operator acting on the subspace {$\mathcal B$}.

\begin{lemma}\label{lem:cov_subspace}
Let $(u_t^{(j)},t\ge0)_{j=1,\dots,J}$ be the solution of \eqref{eq:cont_deterministicEKI} initialized by a linearly independent ensemble $(u_0^{(j)})_{j=1,\dots,J}$ such that
\begin{equation}\label{eq:smallest_eigenvalue}
\zeta_0 = \min_{z\in {\mathcal B},\ \|z\|=1}\ \langle z,C(u_0)z\rangle > 0.
\end{equation}
 For each $z\in {\mathcal B}$ with $\|z\|=1$ it holds true that
\begin{equation*}
\langle z, C(u_t) z\rangle \ge \frac{1}{m t+\zeta_0^{-1}},
\end{equation*}
where $m=2 (c_{\mathrm{lip}}^2\lambda_{\max}V_e(0)+\sigma_{\max})>0$ depends on the smallest and largest eigenvalue $\sigma_{\max}$ and $\sigma_{\min}$ of $C_0^{-1}$,  the largest eigenvalue $\lambda_{\max}$ of $\Gamma^{-1}$ and the Lipschitz constant $c_{\mathrm{lip}}$ of $G$.
\end{lemma}

\begin{proof}
We note that the following proof is closely related to the proof of Theorem~3.5 in \cite{CST19}, where the authors derived an upper bound on the sample covariance matrix.
The time evolution of the sample covariance is given by
\begin{align*}
\frac{{\mathrm d} C(u_t)}{{\mathrm d}t} &= \frac{1}{J}\sum_{j=1}^J \frac{{\mathrm d}e_t^{(j)}}{{\mathrm d}t}\left(e_t^{(j)}\right)^\top + \frac{1}{J}\sum_{j=1}^J e_t^{(j)}\left(\frac{{\mathrm d}e_t^{(j)}}{{\mathrm d}t}\right)^\top\\
& = -2C^{u,G}(u_t)\Gamma^{-1}C^{G,u}(u_t) - 2C(u_t)C_0^{-1}C(u_t).
\end{align*}
Let $z\in {\mathcal B}$ with $\|z\|=1$, then we have that
\begin{align*}
\langle z, \frac{{\mathrm d} C(u_t)}{{\mathrm d}t} z\rangle
& = -2 \|C^{G,u}(u_t)z\|_{\Gamma}^2 - 2\|C(u_t)z\|_{C_0}^2.
\end{align*}

Since $G$ is locally Lipschitz, we obtain with Cauchy-Schwarz
\begin{align*}
\|C^{G,G}(u_t)\|_{\mathrm{HS}} = \Tr(C^{G,G}(u_t)C^{G,G}(u_t))^{1/2} &\le \frac1J\sum_{j=1}^J\|G(u_t^{(j)})-\bar G\|^2 \\ &\le c_{\mathrm{lip}}^2 V_e(t)\le c_{\mathrm{lip}}^2V_e(0)
\end{align*}
and therefore it follows with $\|C^{G,u}(u_t)z\|^2 \le \|C^{G,G}\|_{\mathrm{HS}} \|C(u_t)z\|_{\mathrm{HS}}$ that
\begin{equation*}
\|C^{G,u}(u_t)z\|_{\Gamma}^2\le c_{\mathrm{lip}}^2\lambda_{\max}V_e(0) \|C(u_t) z \|_{C_0}^2.
\end{equation*}
Hence, we have derived
\begin{equation*}
\frac{{\mathrm d} \langle z,C(u_t)z\rangle }{{\mathrm d}t} \ge -2(c_{\mathrm{lip}}^2\lambda_{\max}V_e(0)+\sigma_{\max})\|C(u_t)z\|^2.
\end{equation*}
Next, we consider $C(u_t)$ as operator acting on {$\mathcal B$} and denote its smallest eigenvalue $\zeta_t\ge0$ with unit-norm eigenvector $\varphi(t)\in {\mathcal B}$.  In the following we will prove that $\zeta_t$ is indeed strictly positive.
Since $\|\varphi(t)\|=1$ for all $t\ge0$, we observe that
\begin{equation*}
0 = \frac{{\mathrm{d}}\|\varphi(t)\|^2}{{\mathrm d}t} = 2\langle \varphi(t),\frac{{\mathrm d}\varphi(t)}{{\mathrm d}t}\rangle.
\end{equation*}
Hence, we can describe the time-evolution of the smallest eigenvalue $\zeta_t$ of $C(u_t)$ (initialized with $\zeta_0>0$) through
\begin{align*}
\frac{{\mathrm d}}{{\mathrm d}t}\zeta_t = \frac{{\mathrm d}}{{\mathrm d}t}\langle \varphi(t),C(u_t) \varphi(t)\rangle &= \langle \varphi(t), \frac{{\mathrm d}C(u_t)}{{\mathrm d}t}\varphi(t)\rangle + 2\langle \frac{{\mathrm d}\varphi(t)}{{\mathrm d}t}, C(u_t)\varphi(t)\rangle\\
& = \langle \varphi(t), \frac{{\mathrm d}C(u_t)}{{\mathrm d}t} \varphi(t)\rangle + 2\zeta_t\langle \frac{{\mathrm d}\varphi(t)}{{\mathrm d}t}, \varphi(t)\rangle\\
&= \langle \varphi(t), \frac{{\mathrm d}C(u_t)}{{\mathrm d}t} \varphi(t)\rangle\\
&\ge -2 (c_{\mathrm{lip}}^2\lambda_{\max}V_e(0)+\sigma_{\max})\|C(u_t)\varphi(t)\|^2\\
&\ge -2(c_{\mathrm{lip}}^2\lambda_{\max}V_e(0)+\sigma_{\max})\zeta_t^2.
\end{align*}
With the above computation we obtain the lower bound on the smallest eigenvalue of $C(u_t)$ as operator acting on {$\mathcal B$} by
\begin{equation*}
\zeta_t \ge \frac{1}{2(c_{\mathrm{lip}}^2\lambda_{\max}V_e(0)+\sigma_{\max}) t+ \zeta_0^{-1}}>0,
\end{equation*}
and the assertion follows with 
\begin{equation}\label{eq:const_m}
m:=2 (c_{\mathrm{lip}}^2\lambda_{\max}V_e(0)+\sigma_{\max})>0.
\end{equation}
\end{proof}

\subsection{Gradient approximation via Taylor expansion}
In the following we will quantify how the TEKI flow \eqref{eq:cont_deterministicTEKI} approximates gradients and hence can be viewed as gradient flow. 
\begin{lemma}\label{lem:grad_approx}
Suppose that Assumption~\ref{ass:Taylor} is satisfied. Given a particle system $(u^{(j)})_{j=1,\dots,J}$ it holds true that 
\[ \|C^{u,G}(u)\Gamma^{-1}(G(u^{(j)})-y) - C(u) \nabla\Phi(u^{(j)})\| \le b_1 J \sqrt{\Phi(u^{(j)})}
\left(\frac1J\sum_{k=1}^J\|u^{(k)}-\bar u\|^2\right)^{3/2},\]
for some $b_1$ independent of $J$.
\end{lemma}
\begin{proof}
We apply \eqref{eq:linearization} and using bi-linearity of the inner product we obtain
\begin{align*}
 &C^{u,G}(u)\Gamma^{-1}(G(u^{(j)})-y) = \frac{1}{J}\sum_{k=1}^J \langle G(u^{(k)})-\bar G, G(u^{(j)})-y\rangle_\Gamma (u^{k}-\bar u)\\
 &= \frac{1}J \sum_{k=1}^J \langle G(u^{(j)}) + {\mathrm D}G(u^{(j)})(u^{(k)}- u^{(j)}) + {\mathrm{Res}}(u^{(k)},u^{(j)}),G(u^{(j)})-y\rangle_{\Gamma}(u^{(k)}-\bar u) \\ &\quad - \frac1J\sum_{l=1}^J \langle G(u^{(j)}) + {\mathrm D}G(u^{(j)}) (u^{(l)}-u^{(j)}) + {\mathrm{Res}}(u^{(l)},u^{(j)}) ,G(u^{(j)})-y\rangle_{\Gamma}(u^{(k)}-\bar u)\\
 &= \frac{1}{J}\sum_{k=1}^J \langle {\mathrm D}G(u^{(j)})(u^{(k)}-\bar u),G(u^{(j)})-y\rangle_\Gamma(u^{(k)}-\bar u)\\  &\quad + \frac{1}{J}\sum_{k=1}^J \langle {\mathrm{Res}}(u^{(k)},u^{(j)})-\overline{\mathrm{Res}}^{(j)},G(u^{(j)})-y\rangle_{\Gamma}(u^{(k)}-\bar u)\\
 &= C(u)\nabla\Phi(u^{(j)}) + \frac{1}{J}\sum_{k=1}^J \langle {\mathrm{Res}}(u^{(k)},u^{(j)})-\overline{\mathrm{Res}}^{(j)},G(u^{(j)})-y\rangle_\Gamma(u^{(k)}-\bar u),
\end{align*}
where we have defined
\[\overline{\mathrm{Res}}^{(j)} = \frac{1}{J}\sum_{l=1}^J{\mathrm{Res}}(u^{(l)},u^{(j)}).\]
It follows with Assumption~\ref{ass:Taylor} and Cauchy-Schwarz that
\begin{align*}
 &\|C^{u,G}(u)\Gamma^{-1}(G(u^{(j)})-y)-C(u)\nabla\Phi(u^{(j)}) \|\\  &\le b_{\mathrm{res}}\|\Gamma^{-1/2}\| \frac{1}{J}\sum_{k=1}^J 4\left(\|u^{(k)}-u^{(j)}\|^2+\frac1J\sum_{l=1}^J \|u^{(l)}-u^{(j)}\|^2\right)\|u^{(k)}-\bar u\|\|G(u^{(j)})-y\|_{\Gamma}\\
 &\le b_{\mathrm{res}}\|\Gamma^{-1/2}\| \frac{1}{J}\sum_{k=1}^J 4\left((1+\frac1J)\sum_{l=1}^J \|u^{(l)}-u^{(j)}\|^2\right)\|u^{(k)}-\bar u\|\|G(u^{(j)})-y\|_{\Gamma}\\
 &\le b_{\mathrm{res}}\|\Gamma^{-1/2}\| (8+\frac{2}{J})J\|G(u^{(j)})-y\|_\Gamma\left(\frac1J\sum_{k=1}^J\|u^{(k)}-\bar u\|^2\right)^{3/2},
\end{align*}
where we have used that
\[\|u^{(l)}-u^{(j)}\|^2 \le 2 \|u^{(l)}-\bar u\|^2 + 2 \|u^{(j)}-\bar u\|^2. \]

\end{proof}

We are now ready to prove Proposition~\ref{prop:gradientflow_structure} in order to quantify the gradient flow structure of TEKI.

\begin{proof}[Proof of Proposition~\ref{prop:gradientflow_structure}]
We apply chain rule for deriving the time-evolution of $V_{\Phi_R}(t) := \frac1J \sum_{j=1}^J\Phi_R(u_t^{(j)})$ as
\begin{align*}
\frac{{\mathrm d}V_{\Phi_R}(t)}{{\mathrm d}t} &= \frac{{\mathrm d}V_{\Phi_R}(t)}{{\mathrm d}t} = \frac1J\sum_{j=1}^J \langle \nabla\Phi_R(u_t^{(j)}),\frac{{\mathrm d}u_t^{(j)}}{{\mathrm d}t}\rangle\\
&= -\frac{1}{J}\sum_{j=1}^J \langle \nabla \Phi_R(u_t^{(j)}), C(u_t)\nabla\Phi_R(u_t^{(j)})\rangle \\
&\quad + \frac{1}{J}\sum_{j=1}^J \langle  \nabla\Phi_R(u_t^{(j)}), C^{u,G}(u_t)\Gamma^{-1}(G(u_t^{(j)})-y)-C(u_t)\nabla \Phi(u_t^{(j)})\rangle.
\end{align*}
We apply Cauchy-Schwarz inequality and Lemma~\ref{lem:grad_approx} such that
\begin{align*}
\frac{{\mathrm d}V_{\Phi_R}(t)}{{\mathrm d}t} &\le  -\frac{1}{J}\sum_{j=1}^J \langle \nabla \Phi_R(u_t^{(j)}), C(u_t)\nabla\Phi_R(u_t^{(j)})\rangle\\
&\quad + b_1 J V_e(t)^{3/2}\frac{1}{J}\sum_{j=1}^J\|\nabla\Phi_R(u_t^{(j)})\|\sqrt{\Phi(u_t^{(j)})}\\ 
&\le  -\frac{1}{J}\sum_{j=1}^J \langle \nabla \Phi_R(u_t^{(j)}), C(u_t)\nabla\Phi_R(u_t^{(j)})\rangle\\
&\quad + b_1 J V_e(t)^{3/2}\frac{1}{J}\sum_{j=1}^J\left(\frac{1}{2}\|\nabla\Phi_R(u_t^{(j)})\|^2+ \Phi(u_t^{(j)}) \right).
\end{align*}
The claim follows with $\Phi(x)\le \Phi_R(x)$.
\end{proof}

{Considering the upper bound on the derived gradient approximation in Lemma~\ref{lem:grad_approx}, we observe that it is increasing in the ensemble size $J$. Although we do not know whether this bound is sharp, it suggests that the gradient approximation might be less well-behaved in the large ensemble size regime.  To alleviate this issue, we expect improvement when considering localised preconditioning covariance matrices $C^{u,G}(u^{(j)})$ for each particle $u^{(j)}$. For the localised covariance we include more weights on particles in the neighborhood around $u^{(j)}$ and less weight for particles far away. This type of localisation has been considered for the ensemble Kalman sampler, where we do not expect the particle system to collapse \cite{RW2021}. The convergence analysis of (T)EKI under localisation is work in progress.}

\subsection{TEKI as constrained optimization method}

Before proving our main convergence result, we need to quantify the behavior of the preconditioning effect through the sample covariance, i.e.~what does it mean to consider the direction $C(u_t)\nabla\Phi_R(u_t^{(j)})$ instead of the steepest descent direction given through $\nabla\Phi_R(u_t^{(j)})$. We first note that once $C(u_t)$ can be identified as strictly positive definite operator on the whole parameter space $\mathcal X$, the scheme  can be viewed as global optimization method. 
However, in general it is not true that $C(u_t)$ describes a strictly positive definite operator on the parameter space $\mathcal X$ and indeed, $C(u_t)$ has at most rank $\min(n_x,J-1)$ and hence is positive semi-definite in case $J-1<n_x$. While $C(u_t)$ is only positive semi-definite on the whole space $\mathcal X$, we have verified in Lemma~\ref{lem:cov_subspace} that $C(u_t)$ is a strictly positive definite operator acting on the subspace {$\mathcal B$} spanned through the initial ensemble. We will make use of this property in order to quantify the TEKI flow \eqref{eq:cont_deterministicTEKI} as optimizer for the constrained optimization problem
{
\begin{equation*}
\min_{x\in \mathcal X}\ \Phi_R(x),\quad x\in\mathcal B= u_0^\perp + \spn\{e_0^{(j)} \mid j=1,\dots,J\}. 
\end{equation*}
}
Since we have assumed the finite dimensional setting $\mathcal X=\mathbb R^{n_x}$, we reformulate this problem as equality constrained optimization problem
\begin{equation}\label{eq:constr_opti2}
\min_{x\in \mathbb R^{n_x}}\ \Phi_R(x),\quad \text{s.t.}\quad h_i(x) = \langle q^{(i)},x\rangle = 0,\ i=J+1,\dots, n_x,
\end{equation}
where {$q^{(i)}\perp e_0^{(j)}$ and $q^{(i)}\perp u_0^\perp$} for all $j=1,\dots,J$ and $i=J+1,\dots,n_x$.  Since under Assumption~\ref{ass:smoothness} $\Phi_R$ is assumed to be strongly convex, there exists a unique global solution $u_\ast$ of \eqref{eq:constr_opti2} \cite[Proposition~2.1.1]{DPB15}.  By the Lagrange multiplier theorem \cite[Proposition~3.1.1.]{DPB15} there exist $\kappa_{J+1}^\ast,\dots,\kappa_{n_x}^\ast\in\mathbb R$
\begin{equation}\label{eq:nec_conditions}
\nabla\Phi_R(u_\ast) + \sum_{i=J+1}^{n_x}\kappa_i \nabla h_i(x_\ast) = 0.
\end{equation}
We will make use of this necessary condition in order to derive the following constrained PL-type inequality.
\begin{lemma}\label{lem:constr_PL}
Suppose Assumption~\ref{ass:smoothness} is satisfied and let $x_\ast$ be the unique global minimizer of \eqref{eq:constr_opti2}. Then there exists $\nu>0$ such that for all $x\in{\mathcal B}$ it holds true that
\begin{equation*}
\nu\|\nabla \Phi_R(x)\|^2\ge \Phi_R(x)-\Phi_R(x_\ast).
\end{equation*}
Moreover, let {$\mathcal P_{\mathcal B}$} be the orthogonal projection onto {$\mathcal B$}. Then it holds true that
\begin{equation*}
{\nu\|\mathcal P_{\mathcal B}\nabla \Phi_R(\mathcal P_{\mathcal B} x)\|^2\ge \Phi_R(\mathcal P_{\mathcal B} x) - \Phi_R(x_\ast)}
\end{equation*}
for all $x\in\mathbb R^{n_x}$.
\end{lemma}
\begin{proof}
Consider some arbitrary {$x\in \mathcal B$}. By $L$-smoothness and the necessary optimality condition \eqref{eq:nec_conditions} we have that
\begin{align*}
\Phi_R(x)&\le \Phi_R(x_\ast) + \langle \nabla\Phi_R(x_\ast),x-x_\ast\rangle + \frac{L}{2}\|x-x_\ast\|^2\\
& = \Phi_R(x_\ast) - \sum_{i=J+1}^{n_x} \kappa_i \langle \nabla h_i(x_\ast),x-x_\ast\rangle + \frac{L}{2}\|x-x_\ast\|^2\\
& = \Phi_R(x_\ast) - \sum_{i=J+1}^{n_x} \kappa_i \langle q^{(i)},x-x_\ast\rangle + \frac{L}{2}\|x-x_\ast\|^2\\
& = \Phi_R(x_\ast) + \frac{L}{2}\|x-x_\ast\|^2,
\end{align*}
where we have used $x,x_\ast\perp q^{(i)}$ for all $i=J+1,\dots,n_x$.  
On the other side with $\mu$-strong convexity and Cauchy-Schwarz we obtain
\[\frac{\mu}{2} \|x_\ast-x\| \le \|\nabla \Phi(x)\|\]
leading to
\[\Phi_R(x)-\Phi_R(x_\ast) \le \frac{L}{2} \|x-x_\ast\|^2 \le \frac{L}{\mu}\|\nabla\Phi_R(x)\|^2.\]
For the second claim, consider $x\in\mathbb R^{n_x}$. By $\mu$-strong convexity it follows
\[\langle {\mathcal P_{\mathcal B}} \nabla \Phi_R(x), {\mathcal P_{\mathcal B}}(x-x_\ast)\rangle = \langle \nabla \Phi_R(x),  {\mathcal P_{\mathcal B}}x-x_\ast)\rangle \ge \Phi_R({\mathcal P_{\mathcal B}} x)-\Phi_R(x_\ast)+\frac{\mu}{2} \|{\mathcal P_{\mathcal B}} x-x_\ast\|^2. \]
The assertion follows again by applying Cauchy-Schwarz inequality and $L$-smoothness.
\end{proof}

\section{Proof of Theorem~\ref{thm:main}}\label{sec:proof_main}

We are now ready to present the proof of our main convergence result.

\begin{proof}[Proof of Theorem~\ref{thm:main}]
Let 
$\mathcal B= u_0^\perp + \spn\{e_0^{(j)} \mid j=1,\dots,J\}$
and let {$\mathcal P_{\mathcal B}$} be the orthogonal projection onto {$\mathcal B$}.  Recall that from Proposition~\ref{prop:gradientflow_structure} we have that
\begin{align*}
 \frac{{\mathrm d}\frac1J\sum_{j=1}^J\left(\Phi_R(u_t^{(j)})\right)}{{\mathrm d}t} &\le -\frac1J\sum_{j=1}^J \langle \nabla \Phi_R(u_t^{(j)}), C(u_t) \nabla \Phi_R(u_t^{(j)})\rangle\\ &\quad+  b_1 J V_e(t)^{3/2}\frac{1}{J}\sum_{j=1}^J\left(\frac{1}{2}\|\nabla\Phi_R(u_t^{(j)})\|^2+ \Phi_R(u_t^{(j)}) \right)\\
&\le -\frac1J\sum_{j=1}^J \langle \nabla \Phi_R(u_t^{(j)}), C(u_t)\nabla \Phi_R(u_t^{(j)})\rangle\\ &\quad+ \frac32 b_1 J \left(\frac{1}{\frac{2\sigma_{\min}}{J}t+V_e(0)^{-1}}\right)^{3/2}\frac{1}{J}\sum_{j=1}^J(\Phi_R(u_t^{(j)})-\Phi_R(u_\ast)) \\
&\quad+ \tilde b_1 J \left(\frac{1}{\frac{2\sigma_{\min}}{J}t+V_e(0)^{-1}}\right)^{3/2}\Phi_R(u_\ast)),
\end{align*}
where we have used strong convexity.  Due to the subspace space property in Proposition~\ref{prop:subspace_prop} we obtain $C(u_t) =  \mathcal P_{\mathcal B}C(u_t) \mathcal P_{\mathcal B}$ and with Lemma~\ref{lem:cov_subspace} it holds true that
\begin{align*}
 \langle \nabla \Phi_R(u_t^{(j)}), C(u_t) \nabla \Phi_R(u_t^{(j)})\rangle &=  \langle {\mathcal P_{\mathcal B}}\nabla \Phi_R(u_t^{(j)}), C(u_t) {\mathcal P_{\mathcal B}}\nabla \Phi_R(u_t^{(j)})\rangle\\ &\ge  \frac{1}{mt+\zeta_0^{-1}}\|{\mathcal P_{\mathcal B}}\nabla\Phi_R(u_t^{(j)}) \|^2,
 \end{align*}
where $m>0$ is defined in \eqref{eq:const_m} and $\zeta_0$ is defined in \eqref{eq:smallest_eigenvalue}.  Hence,  we define $V_{\Phi_R}(t):=\frac{1}{J}\sum_{j=1}^J \Phi_R(u_t^{(j)})-\Phi_R(u_\ast) $ and apply Lemma~\ref{lem:constr_PL} in order to obtain
\begin{align*}
\frac{{\mathrm d}V_{\Phi_R}(t)}{{\mathrm d}t} = \frac{{\mathrm d}\frac1J\sum_{j=1}^J\left(\Phi_R(u_t^{(j)})\right)}{{\mathrm d}t} 
&\le - \frac{1}{\frac{L}{\mu}mt+\zeta_0^{-1}}V_{\Phi_R}(t)\\
&\quad+ \tilde b_1 J \left(\frac{1}{\frac{2\sigma_{\min}}{J}t+V_e(0)^{-1}}\right)^{3/2} V_{\Phi_R}(t) \\
&\quad+ \tilde b_1 J \left(\frac{1}{\frac{2\sigma_{\min}}{J}t+V_e(0)^{-1}}\right)^{3/2}\Phi_R(u_\ast).
\end{align*}
It follows with $\alpha = \frac{L}{\mu}m>0$ that 
\begin{align*}
\frac{{\mathrm d}V(t)}{{\mathrm d}t}\le -\frac{1}{\alpha t+a} V(t) + \left(\frac{z_1}{z_2+z_3t} \right)^{3/2}(V(t)+c)
\end{align*}
for some constants $a,z_1,z_2,z_3,c>0$. By Gronwall's inequality we obtain
\begin{align*}
V(t) &\le V(0) \left(\int_0^t \left(\frac{z_1}{z_2+z_3s} \right)^{3/2}\, {\mathrm d}s\right) \exp\left(\int_0^t \left(\frac{z_1}{z_2+z_3s} \right)^{3/2}\,{\mathrm d}s -\int_0^t \frac{1}{\alpha s+a}\,{\mathrm d}s\right)\\
 &\le \left(\frac{c_1}{t+c_2}\right)^{\frac1\alpha},
\end{align*}
for some constants $c_1,c_2>0$.
\end{proof}

\begin{remark}\label{rem:higher_apprx}
We note that under the relaxed assumption $\|{\mathrm{Res}}(x,y)\| \le b_{\mathrm{res}} \|x-y\|^p$,  for $p\ge 2$ such that $p/2\in\mathbb N$, we obtain a similar bound for the gradient approximation as in Lemma~\ref{lem:grad_approx} given by
\begin{align*}
&\|C^{u,G}(u)\Gamma^{-1}(G(u^{(j)})-y)-C(u)\nabla\Phi(u^{(j)}) \|\\ 
&\le b_{\mathrm{res}}\|\Gamma^{-1/2}\| (8+\frac{2}{J})J(4J)^{p/2}\|G(u^{(j)})-y\|_\Gamma\left(\frac1J\sum_{k=1}^J\|u^{(k)}-\bar u\|^2\right)^{1/2+p/2}.
\end{align*}
Here, we have applied Binomial theorem to imply the inequality 
\[\sum_{j=1}^{J}a_j^p\le \left(\sum_{j=1}^{J} a_j^2\right)^{p/2},\quad a_j\ge0.\]
Since $p\ge2$ the integral $\int_0^t \left(\frac{z_1}{z_2+z_3s} \right)^{1/2+p/2}\,{\mathrm d}s$ in the above proof remains finite and hence the asymptotic convergence result in Theorem~\ref{thm:main} still holds. 
\end{remark}

\section{Stabilization via covariance inflation: A generalized ensemble square root filter for inverse problems}\label{sec:ESRF}

The incorporation of covariance inflation is a data assimilation tool in order to stabilize convergence and often used in practice \cite{JLA09,JLA07,47c5c78ebb8c44ef9d8d5e0d23e23c13}. In the following we present a form of covariance inflation for the TEKI flow for which it turns out to be a generalized formulation of the ensemble square root filter (ESRF) \cite{RC14} applied to inverse problems.  

Through our incorporation of covariance inflation we obtain in approximation a weighted gradient flow structure between the current particle position and the position of the ensemble mean. Roughly speaking, the TEKI flow can then be viewed as approximation of the equation:
\begin{equation}\label{eq:approx_ESRF_flow}
\frac{\mathrm{d}u_t^{(j)}}{{\mathrm d}t} = - C(u_t) \left((1-\rho)\nabla \Phi_R(u_t^{(j)})+\rho\nabla \Phi_R(\bar u_t)\right),
\end{equation}
where $\rho\in[0,1)$ describes the weight of the gradients of the particles and the mean.  Based on \cite{bergemann2010mollified, bergemann2010localization} the continuous time formulation of the ESRF applied to inverse problems has been analysed in \cite{CSW19} for linear forward models and the special case $\rho=1/2$. 
In the following, we will incorporate the covariance inflation and theoretical quantify the approximate gradient flow structure. Finally, we will observe an acceleration in convergence towards the optimal solution restricted to the subspace spanned through the initial ensemble spread.

\begin{figure}[htb!]
\centering
\includegraphics[scale=0.4]{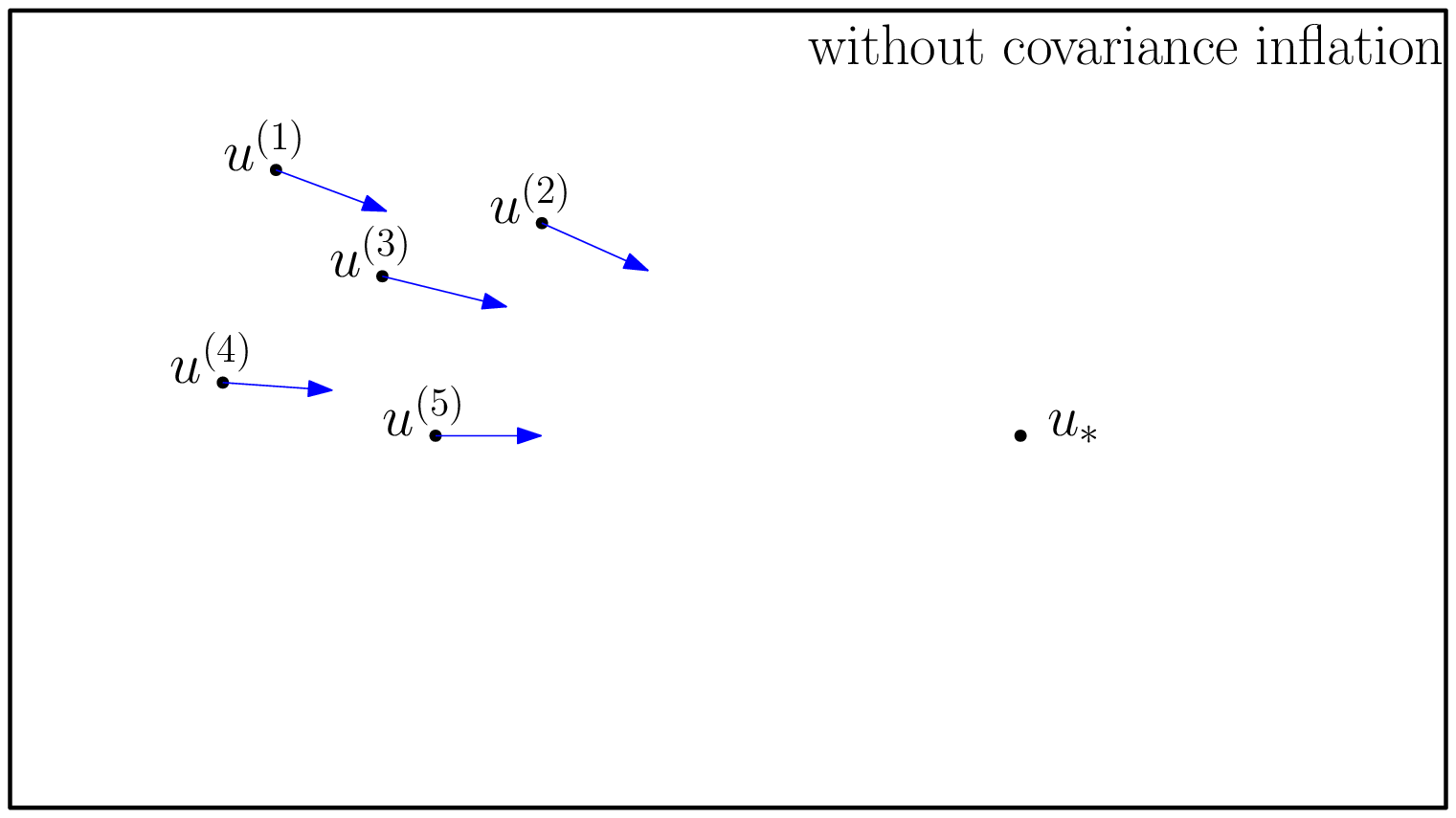}
\includegraphics[scale=0.4]{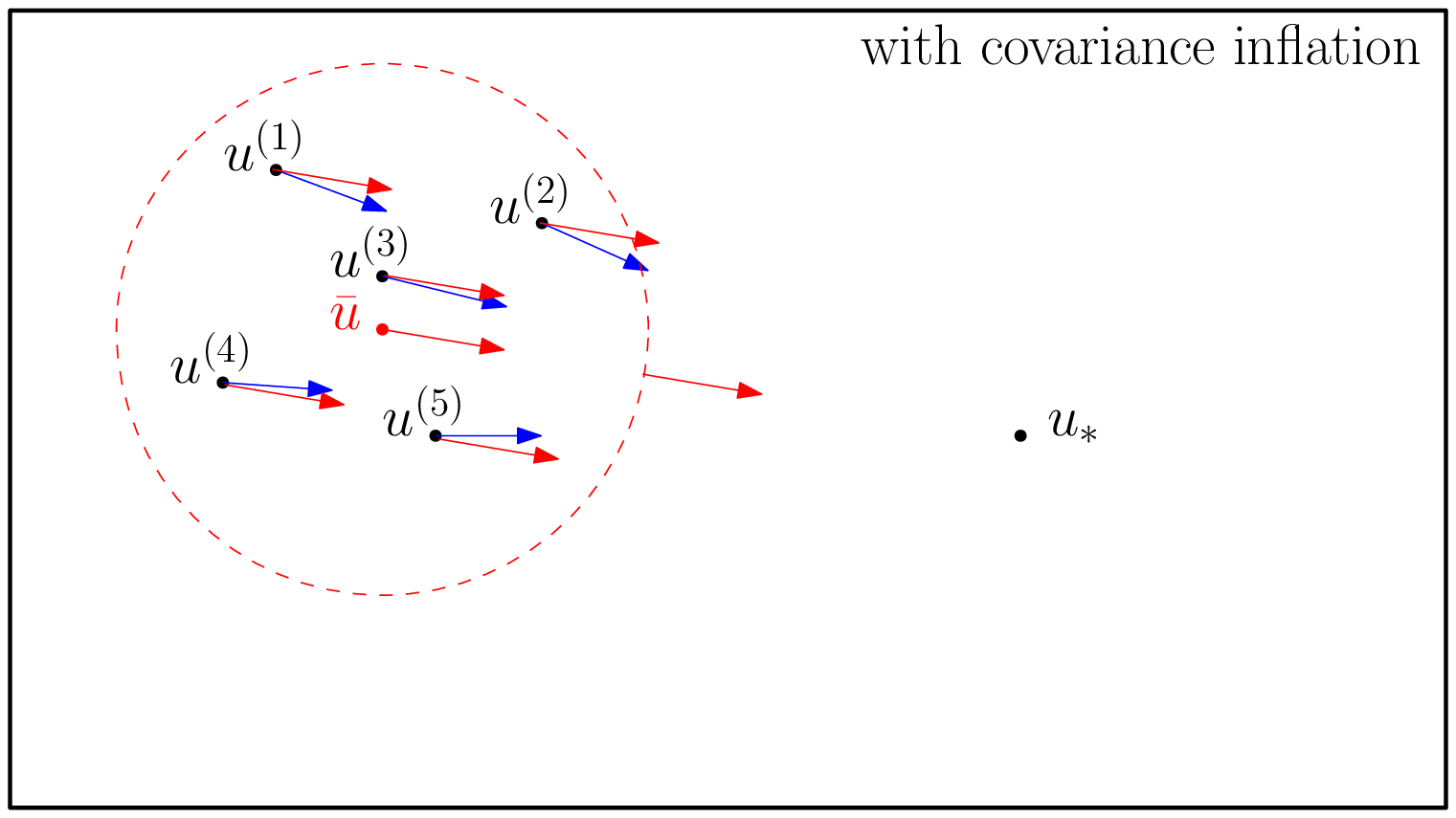}
\caption{Illustration of the effect through the incorporation of the covariance inflation. Each particle is getting moved into its own direction as well as into a joint direction.}
\label{fig:covariance_inflation}
\end{figure}

\subsection{Incorporation of covariance inflation}

The covariance inflation is incorporated through the following modified TEKI flow
\begin{equation}\label{eq:cov_inflation}
\begin{split}
\frac{{\mathrm d}u_t^{(j)}}{{\mathrm d}t} &= C^{u,G}(u_t)\Gamma^{-1}(y-G(u_t^{(j)}))-C(u_t)C_0^{-1}u_t^{(j)}\\  &\quad+ \rho C^{u,G}(u_t)\Gamma^{-1}(G(u_t^{(j)})-\bar G) +\rho C(u_t)C_0^{-1}(u_t^{(j)}-\bar u).
\end{split}
\end{equation}
The idea behind this incorporation comes from the fact that the inflating term does not directly change the dynamical behavior of the particles mean, since the evolution
\begin{equation}\label{eq:cov_inflation_mean}
\frac{{\mathrm d}\bar u_t}{{\mathrm d}t} = C^{u,G}(u_t)\Gamma^{-1}(y-\bar G_t)-C(u_t)C_0^{-1}\bar u_t
\end{equation}
coincides with the unmodified dynamic (i.e.~$\rho=0$).  However the inflating term decreases the force of collapsing the particle system and hence leads to a slower ensemble collapse without changing the collapse rate $1/t$ but decreasing the constant.  For the interpretation from the optimization point of view we rewrite \eqref{eq:cov_inflation} as
\begin{equation}\label{eq:cov_inflation2}
\begin{split}
\frac{{\mathrm d}u_t^{(j)}}{{\mathrm d}t} &= (1-\rho)\left\{C^{u,G}(u_t)\Gamma^{-1}(y-G(u_t^{(j)}))-C(u_t)C_0^{-1}u_t^{(j)}\right\}\\
&\quad +\rho\left\{C^{u,G}(u_t)\Gamma^{-1}(y-\bar G_t)-C(u_t)C_0^{-1}\bar u_t\right\}.
\end{split}
\end{equation}
motivating the modified TEKI flow as approximation of \eqref{eq:approx_ESRF_flow}.  The movement of the particle system through \eqref{eq:cov_inflation2} is described through two forces.  The first force is given by 
\[C^{u,G}(u_t)\Gamma^{-1}(y-G(u_t^{(j)}))-C(u_t)C_0^{-1}u_t^{(j)} \]
moving each particle in its own direction and leading to a collapse of the ensemble and hence to a better gradient approximation as we have already seen in Lemma~\ref{lem:grad_approx}. The second force 
\[C^{u,G}(u_t)\Gamma^{-1}(y-\bar G_t)-C(u_t)C_0^{-1}\bar u_t \]
moves the whole particle system into the same direction without changing the spread and hence does not improve the gradient approximation. However, it accelerates the overall convergence of the optimization process since it slows down the degeneration of the preconditioning through the sample covariance matrix.  We will start the theoretical analysis by considering the ensemble spread and its collapse behavior. 
{We first note, that unique existence of solutions of \eqref{eq:cov_inflation} can be verified and the subspace property Proposition~\ref{prop:subspace_prop} is obviously satisfied.
\begin{proposition}
Let $(u_0^{(j)})_{j=1,\dots,J}$ be the linearly independent initial ensemble and define the subspace
\[\cB = u_0^\perp + \spn\{e_0^{(j)} \mid j=1,\dots,J\},  \]
where $u_0^\perp = \bar u_0 - \mathcal P_E \bar u_0$,  where $\mathcal P_E= E(E^\top E)^{-1}E^T$ with $E=((e_0^{(1)})^\top,\dots,(e_0^{(J)})^\top)\in\mathbb R^{n_x\times J}$  is the orthogonal projection on the subspace $\spn\{e_0^{(j)} \mid j=1,\dots,J\}$. Then for all $T\ge 0$ and $\rho\in[0,T)$ the unique solution $(u_t^{(j)}, t\in[0,T])$  of the flow \eqref{eq:cov_inflation} exists and remains in $\mathcal B$, i.e.~
$u_t^{(j)} \in \mathcal B$ for all $t\in[0,T]$ and all $j=1,\dots,J$.
\end{proposition}
}

The following result states, that the ensemble collapse occurs with reduced rate through the factor $1-\rho\in(0,1]$.
\begin{lemma}\label{lem:ensemble_collapse_ESRF}
Let $\rho\in[0,1)$ and $(u_t^{(j)},t\ge0)_{j=1,\dots,J}$ be the solution of \eqref{eq:cov_inflation} initialized by a linearly independent ensemble $(u_0^{(j)})_{j=1,\dots,J}$. Then the mapping $V_e:\mathbb R_+\to \mathbb R_+$ with $t\mapsto \frac{1}{J}\sum_{j=1}^J\|e_t\|^2$
is bounded by
\[V_e(t) \le \frac{1}{\frac{2(1-\rho)\sigma_{\min}}{J} t + V_e(0)^{-1}},\]
where $\sigma_{\min}$ denotes the smallest eigenvalue of $C_0^{-1}$.
\end{lemma}
\begin{proof}
The evolution of $V_e(t)$ is given by
\begin{align*}
 \frac{{\mathrm d}\frac1J\sum_{j=1}^J \|u_t^{(j)}-\bar u_t\|^2}{{\mathrm d}t}
					&= -(1-\rho)2\|C^{u,G}(u_t)\Gamma^{-1/2}\|_{\mathrm{HS}}\\ &\quad-\frac{(1-\rho)2}{J^2} \sum_{j,k=1}^J \langle u_t^{(j)}-\bar u_t,u_t^{(k)}-\bar u_t\rangle \langle u_t^{(k)}-\bar u_t,C_0^{-1}(u_t^{(j)}-\bar u_t)\rangle\\
					&\le -\frac{2(1-\rho)\sigma_{\min}}{J} \left(\frac1J\sum_{j=1}^J \|u_t^{(j)}-\bar u_t\|^2\right)^2,
\end{align*}
and the assertion follows again by Lyapunov-type argument.
\end{proof}

Similarly, we can derive the lower bound of the ensemble collapse.
\begin{lemma}\label{lem:cov_subspace_ESRF}
Let $\rho\in[0,1)$ and $(u_t^{(j)},t\ge0)_{j=1,\dots,J}$ be the solution of \eqref{eq:cov_inflation} initialized by a linearly independent ensemble $(u_0^{(j)})_{j=1,\dots,J}$ such that
\begin{equation*}
\zeta_0 = \min_{z\in {\mathcal B},\ \|z\|=1}\ \langle z,C(u_0)z\rangle > 0.
\end{equation*}
 For each $z\in {\mathcal B}$ with $\|z\|=1$ it holds true that
\begin{equation*}
\langle z, C(u_t) z\rangle \ge \frac{1}{(1-\rho)m t+\zeta_0^{-1}},
\end{equation*}
where $m=2(c_{\mathrm{lip}}^2\lambda_{\max}V_e(0)+\sigma_{\max})>0$ depends on the smallest and largest eigenvalue $\sigma_{\max}$ and $\sigma_{\min}$ of $C_0^{-1}$,  the largest eigenvalue $\lambda_{\max}$ of $\Gamma^{-1}$ and the Lipschitz constant $c_{\mathrm{lip}}$ of $G$.
\end{lemma}

\begin{proof}
The time evolution of the sample covariance is given by
\begin{align*}
\frac{{\mathrm d} C(u_t)}{{\mathrm d}t} &= \frac{1}{J}\sum_{j=1}^J \frac{{\mathrm d}e_t^{(j)}}{{\mathrm d}t}\left(e_t^{(j)}\right)^\top + \frac{1}{J}\sum_{j=1}^J e_t^{(j)}\left(\frac{{\mathrm d}e_t^{(j)}}{{\mathrm d}t}\right)^\top\\
& = -2(1-\rho)C^{u,G}(u_t)\Gamma^{-1}C^{G,u}(u_t) - 2(1-\rho)C(u_t)C_0^{-1}C(u_t)
\end{align*}
and similarly, for $z\in {\mathcal B}$ with $\|z\|=1$ we obtain
\begin{align*}
\langle z, \frac{{\mathrm d} C(u_t)}{{\mathrm d}t} z\rangle
& = -2(1-\rho) \|C^{G,u}(u_t)z\|_{\Gamma}^2 - 2(1-\rho)\|C(u_t)z\|_{C_0}^2
\end{align*}
and therefore
\begin{equation*}
\frac{{\mathrm d} \langle z,C(u_t)z\rangle }{{\mathrm d}t} \ge -2(1-\rho)(c_{\mathrm{lip}}^2\lambda_{\max}V_e(0)+\sigma_{\max})\|C(u_t)z\|^2.
\end{equation*}
We consider again the smallest eigenvalue $\zeta_t\ge0$ of $C(u_t)$ restricted to {$\mathcal B$} with unit-norm eigenvector $\varphi(t)\in {\mathcal B}$. For the time-evolution of $\zeta_t$ we obtain
\[
\frac{{\mathrm d}}{{\mathrm d}t}\zeta_t \ge -2(1-\rho)(c_{\mathrm{lip}}^2\lambda_{\max}V_e(0)+\sigma_{\max})\zeta_t^2,
\]
such that the assertion follows.
\end{proof}

Our convergence result will be based on the mean of the particle system. Therefore we will need the following extension of Lemma~\ref{lem:grad_approx}.
\begin{lemma}\label{lem:grad_approx_mean}
Suppose that Assumption~\ref{ass:Taylor} is satisfied. Given a particle system $(u^{(j)})_{j=1,\dots,J}$ it holds true that 
\[\|C^{u,G}(u)\Gamma^{-1}(\bar G-y) - C(u) \nabla\Phi(\bar u)\| \le b_1 J \sqrt{\Phi(\bar u)}
\left(\frac1J\sum_{k=1}^J\|u^{(k)}-\bar u\|^2\right)^{3/2},\]
for some $b_1$ independent of $J$.
\end{lemma}
\begin{proof}
We apply triangular inequality 
\begin{align*}
\|C^{u,G}(u)\Gamma^{-1}(\bar G-y) - C(u) \nabla\Phi(\bar u)\| &\le \|C^{u,G}(u)\Gamma^{-1}(\bar G-G(\bar u))\|\\ &\quad + \|C^{u,G}(u)\Gamma^{-1}(G(\bar u)-y) - C(u) \nabla\Phi(\bar u)\|.
\end{align*}
Using Lipschitz continuity and the inequality $\|C^{u,G}\|_{\mathrm{HS}}\le \sqrt{\|C^{u,u}\|_{\mathrm{HS}}\|C^{G,G}\|_{\mathrm{HS}}}$ we can bound the first term 
\[ \|C^{u,G}(u)\Gamma^{-1}(\bar G-G(\bar u))\| \le \|\Gamma^{-1}\|c_{\mathrm{lip}}^2  \left(\frac1J\sum_{k=1}^J\|u^{(k)}-\bar u\|^2\right)^{3/2}.\]
The second term can be bounded similarly to Lemma~\ref{lem:grad_approx} and the assertion follows.
\end{proof}

\begin{theorem}\label{thm:main2}
Suppose Assumption~\ref{ass:smoothness} and Assumption~\ref{ass:Taylor} are satisfied, let $(u_t^{(j)},t\ge0)_{j=1,\dots,J}$ solve \eqref{eq:cov_inflation}  with $\rho\in[0,1)$  and  linearly independent initial ensemble $(u_0^{(j)})_{j=1,\dots,J}$ and let $u_\ast\in\mathcal B$ be the unique global minimizer of \eqref{eq:constr_opti}. Then there exist $c_1,c_2>0$ such that
\begin{equation}\label{eq:main_result2}
\Phi_R(\bar u_t)-\Phi_R(u_\ast) \le \left(\frac{c_1}{t+c_2}\right)^{\frac1\alpha},
\end{equation}
where $0<\alpha < (1-\rho)\frac{L}{\mu}(\sigma_{\max}+c_{\mathrm{lip}}\lambda_{\max}\|C(u_0)\|_{\mathrm{HS}})$. 
\end{theorem}
\begin{proof}
With the above presented results in Lemma~\ref{lem:ensemble_collapse_ESRF}, Lemma~\ref{lem:cov_subspace_ESRF} and Lemma~\ref{lem:grad_approx_mean}, the proof is similar to the proof of Theorem~\ref{thm:main}.
\end{proof}

Through the incorporation of covariance inflation we have reduced the ensemble collapse by decreasing the constant of the degeneration rate through the constant inflation parameter $(1-\rho)\in(0,1]$. Resulting from this slow down of the ensemble collapse the convergence rate of TEKI as optimization method can be accelerated by the factor $(1-\rho)$.

\section{Numerical experiment}\label{sec:numerics}

In our numerical experiment we consider a nonlinear inverse problem based on the one-dimensional elliptic boundary-value problem in the form of 
\begin{equation}\label{eq:darcyflow}
\begin{array}{r@{}l}
-\frac{{\mathrm d}}{{\rm d}s}\left(\exp(u(s))\frac{{\mathrm d}}{{\mathrm d}s}p(s)\right) &= 1,   \qquad  s\in[0,1],\\
 p(s) &= 0,  \qquad s\in\{0,1\}.
 \end{array}
\end{equation}
Given discrete noisy observations $y\in\mathbb R^{K}$ of the solution $p\in H_0^1([0,1];\R)$ we consider the task of recovering the unknown parameter $u\in L^\infty([0,1])$.  To be more precise, we consider $y = (\cO\circ S)(u(s))+\eta$, where $\eta\sim\cN(0,\Gamma)$ is additive Gaussian noise,  $S:\ L^\infty([0,1])\to H_0^1([0,1];\R)$ denotes the solution operator of \eqref{eq:darcyflow} and $\cO:\ H_0^1([0,1];\R)\to\R^{K}$ denotes the observation operator, which provides function values at $n_y$ equidistant observation points in $[0,1]$,  i.e.~
$z(\cdot)\in H_0^1([0,1];\R)\mapsto \cO(z(\cdot)) = (z(s_1),\dots,z(s_{K}))^{\top}$, $s_i = \frac{i}{K},\ i=1,\dots,K$. 

For our numerical results we replace $S$ by an numerical solution operator for \eqref{eq:darcyflow} on the grid $\cD\subset[0,1]$ with mesh size $h=2^{-8}$ 
and restrict $u(\cdot)$ to the computational grid $s_l = l\,h$, $l=1,\ldots,2^8-1$, i.e.~we consider the finite dimensional parameter space $\mathcal X=\mathbb R^{n_x}$ with $n_x=2^8$. Our observations are taken at $K = 2^5-1$ observation points and the noise covariance is assumed to be $\Gamma = 0.01\cdot{\mathrm{Id}}\in\R^{K\times K}$.

We set a Gaussian prior $\mathcal N(0,C_0)$ on the unknown parameter with covariance operator $C_0=\beta(-\Delta)^{-1}$,  $\beta = 10$. The underlying ground truth is constructed as realization $u^\dagger\sim\mathcal N(0,C_0)$ and shown in Figure~\ref{fig:reference} with corresponding taken observations.

\begin{figure}[htb!]
\centering
\includegraphics[scale=0.23]{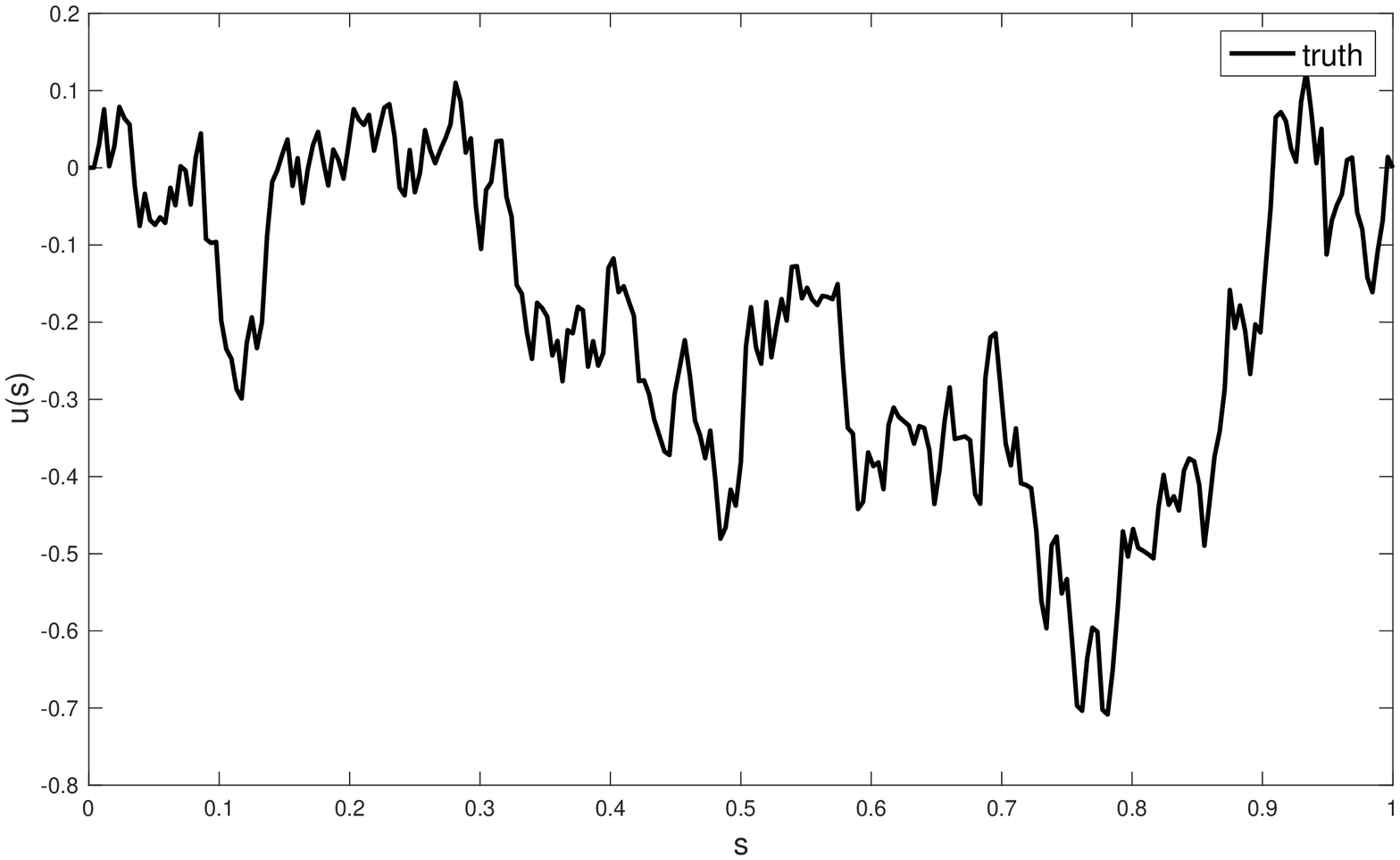}
\includegraphics[scale=0.23]{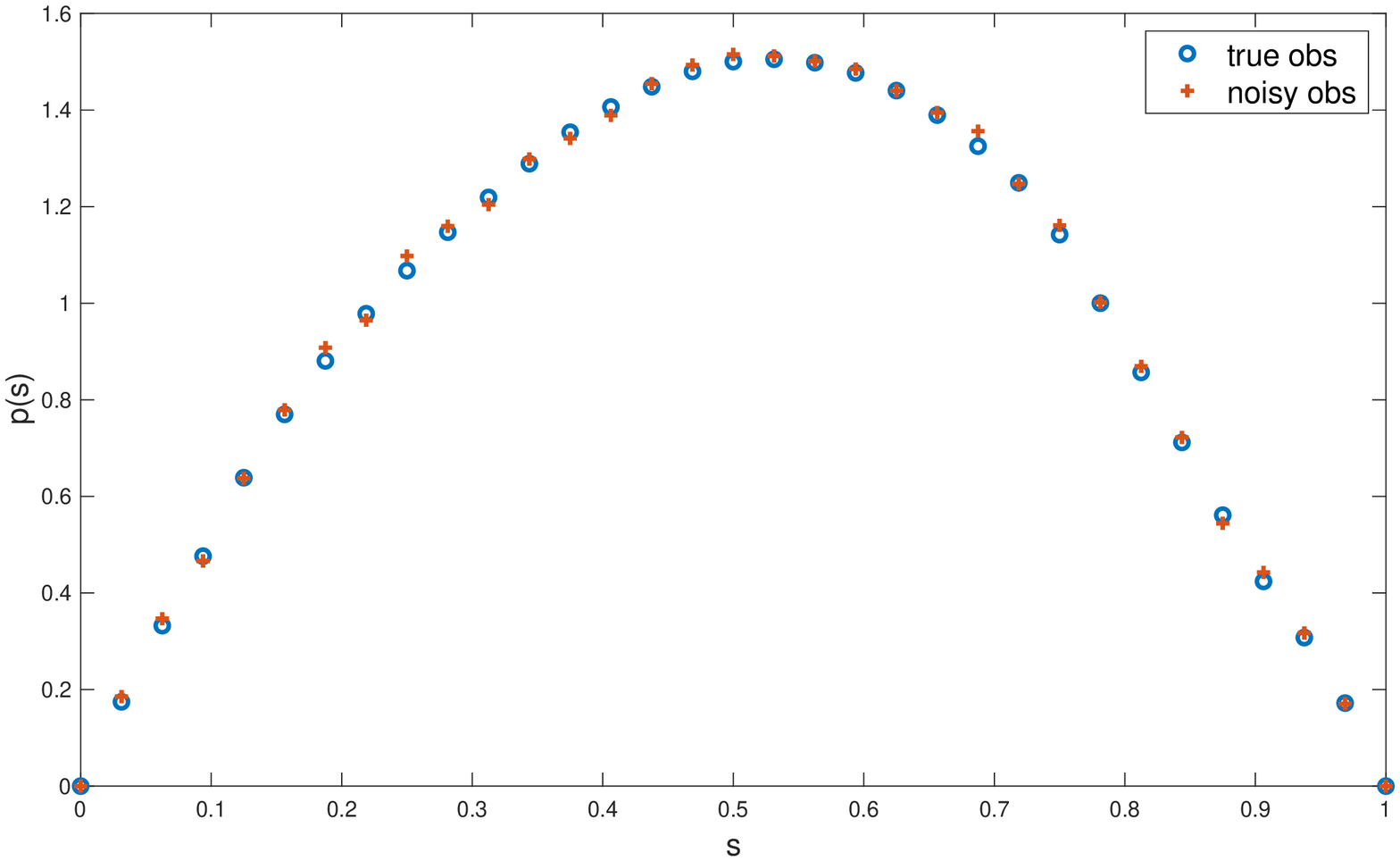}
\caption{Underlying ground truth of the unknown parameter $u$ (left) and the corresponding (noisy) observations (right). }
\label{fig:reference}
\end{figure}

\subsection{Fixed regularization}

{In our first experiment, we consider a fixed regularization and two types of initialization for the implementation of the TEKI flow:}
\begin{itemize}
\item \textbf{Basis initialization}: we choose the initial ensemble of particles based on (ordered) eigenvalues and eigenfunctions $\{\lambda_j,z_j\}_{j\in\N}$ of the covariance operator $C_0$. 
\item \textbf{Random initialization}: we choose the initial ensemble of particles based on i.i.d.~realizations of the Gaussian prior $\mathcal N(0,C_0)$.
\end{itemize}
The TEKI flow with covariance inflation and regularization parameter $\kappa = 0.01^2$
\begin{equation}\label{eq:TEKI_fixed}
\begin{split}
\frac{{\mathrm d}u_t^{(j)}}{{\mathrm d}t} &= C^{u,G}(u_t)\Gamma^{-1}(y-G(u_t^{(j)}))-C(u_t)\kappa C_0^{-1}u_t^{(j)}\\  &\quad+ \rho C^{u,G}(u_t)\Gamma^{-1}(G(u_t^{(j)})-\bar G) +\rho C(u_t)\kappa C_0^{-1}(u_t^{(j)}-\bar u)
\end{split}
\end{equation}
is tested for different choices of inflation scaling $\rho\in\{0,0.25,0.5,0.8\}$ and solved through \texttt{MATLAB} via the function \texttt{ode45} until final time $T=10^7$
with the aim of minimizing 
\begin{equation*}
\Phi_R(u) =  \frac12 \|G(u)-y\|_\Gamma^2 + \frac\kappa2\|u\|_{C_0}^2,\quad u\in\mathbb R^{n_x}.
\end{equation*}
The reference solution $u_\ast^J$ of the constrained optimization problem \eqref{eq:constr_opti2} depending of the ensemble size $J\in\{5,20,50\}$ is approximated with the \texttt{MATLAB} function \texttt{fmincon}. We note that $u_\ast^J$ depends on the initial ensemble and differs for basis and random initialization.

In Figure~\ref{fig:spread_J=50} we show that the ensemble collapses with rate $1/t$ and illustrate the effect of the covariance inflation slowing down the collapse by a constant. This observation verifies our expectation on the behavior from the results in Lemma~\ref{lem:ensemble_collapse_ESRF} and Lemma~\ref{lem:cov_subspace_ESRF}.  Here, we observe no significant difference between basis initialization and random initialization.  We demonstrate our main convergence result presented in Theorem~\ref{thm:main} and Theorem~\ref{thm:main2} through Figure~\ref{fig:Phi_J=5}-\ref{fig:Phi_J=50}, where we observe that the lossfunctions evaluated at the particle mean converge towards the expected minimum restricted to the subspace spanned by the initial ensemble {spread}.  We firstly observe, that increasing the ensemble size, the resulting minimum improves since the initial subspace increases. Secondly, we observe that the resulting estimates for basis function initialization lead to lower values in the lossfunction compared to random initialization. This behavior is expected since the basis functions are weighted through the ordered eigenvalues. Furthermore, in Figure~\ref{fig:Phi_diff_J=50} we illustrate the improved convergence through increasing variance inflation parameter $\rho$. Finally, Figure~\ref{fig:par_est} shows the resulting estimators of the unknown parameter $u$ where we observe smooth estimates for basis function initialization compared to the estimates resulting from random initialization.

\begin{figure}[htb!]
\centering
\includegraphics[scale=0.3]{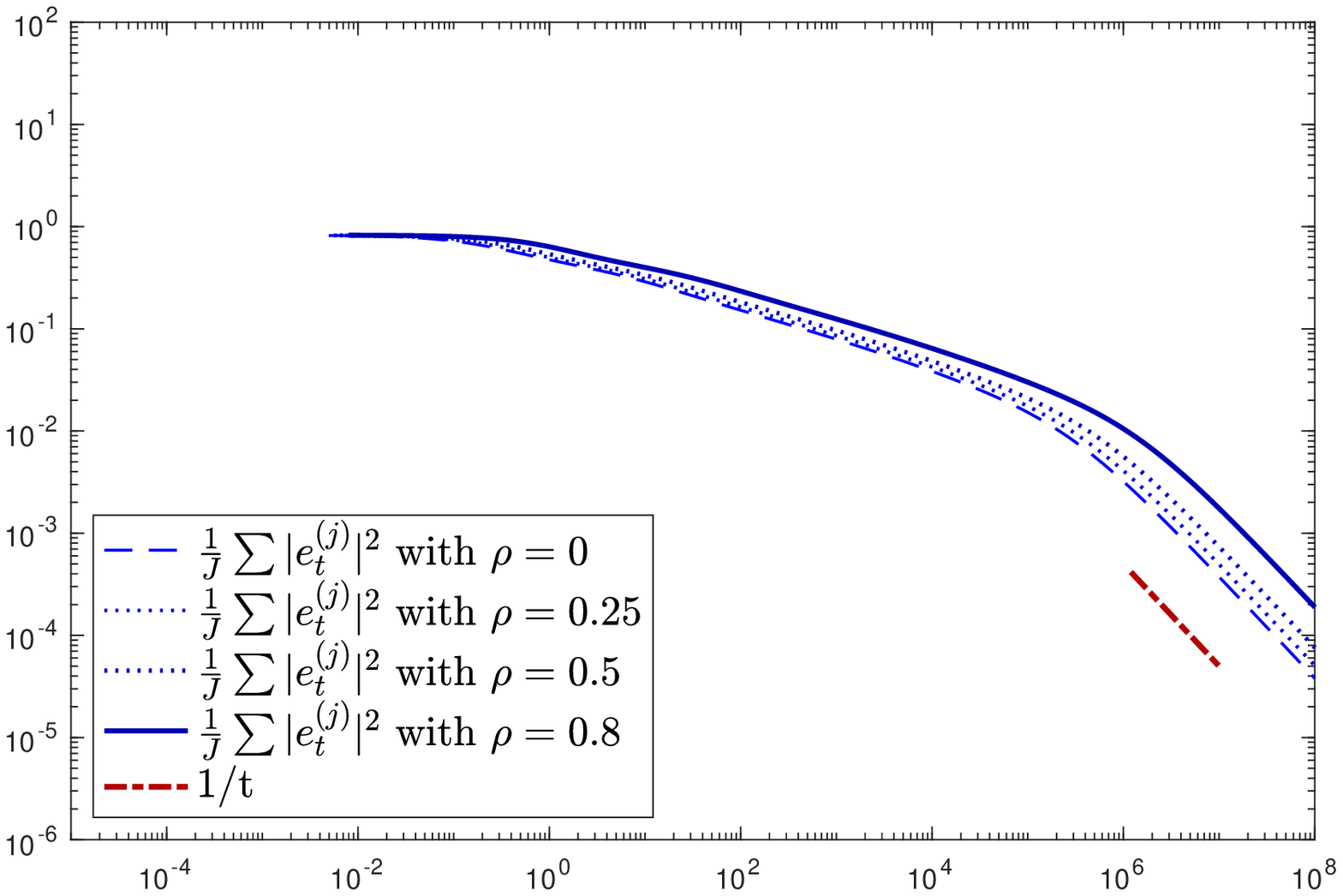}
\includegraphics[scale=0.3]{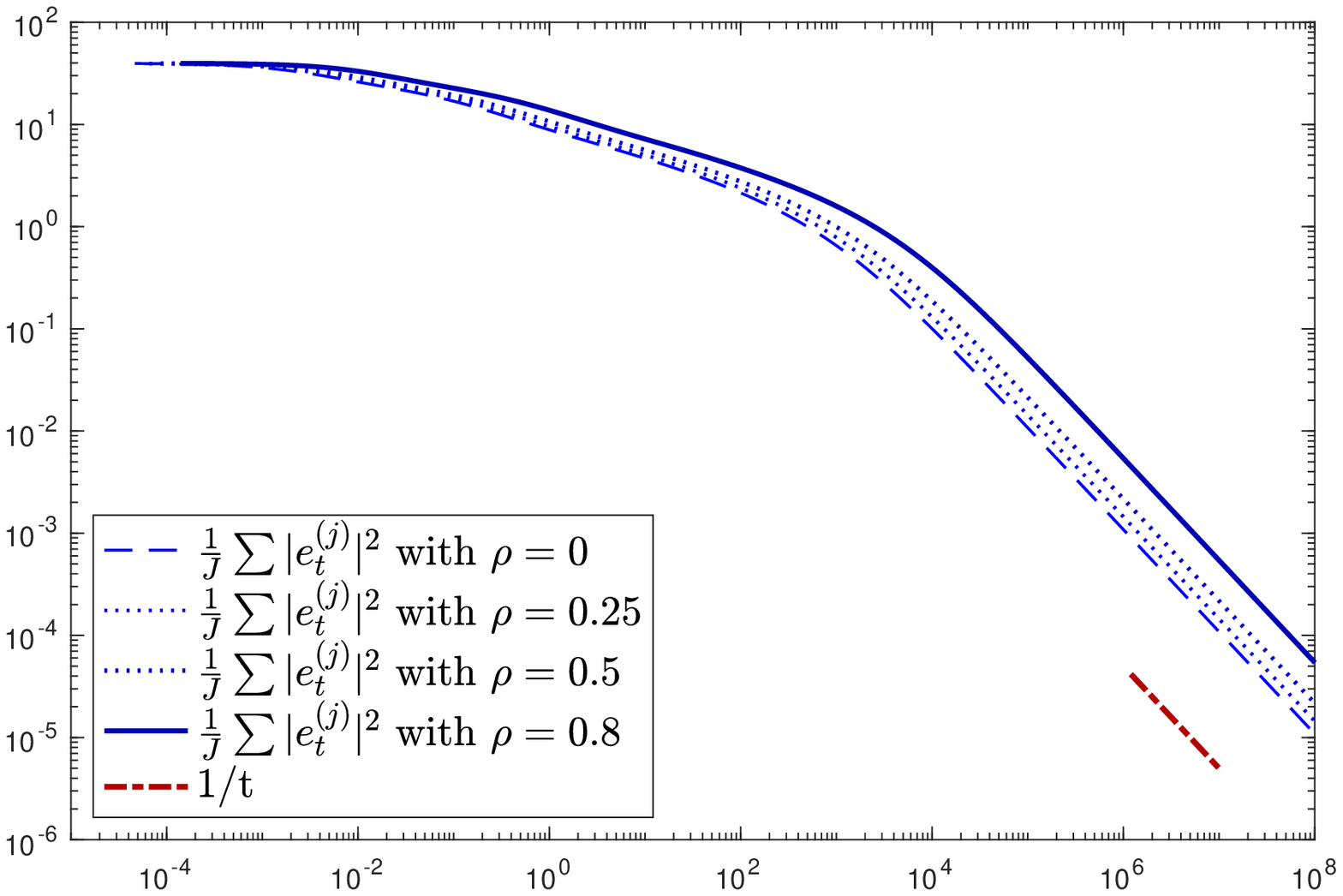}
\caption{Particle spread $\frac1J\sum_{j=1}^J \|e_t^{(j)}\|^2$ of the TEKI flow for basis function initialization (left) and random initialization (right) with ensemble size $J=50$.}
\label{fig:spread_J=50}
\end{figure}

\begin{figure}[htb!]
\centering
\includegraphics[scale=0.3]{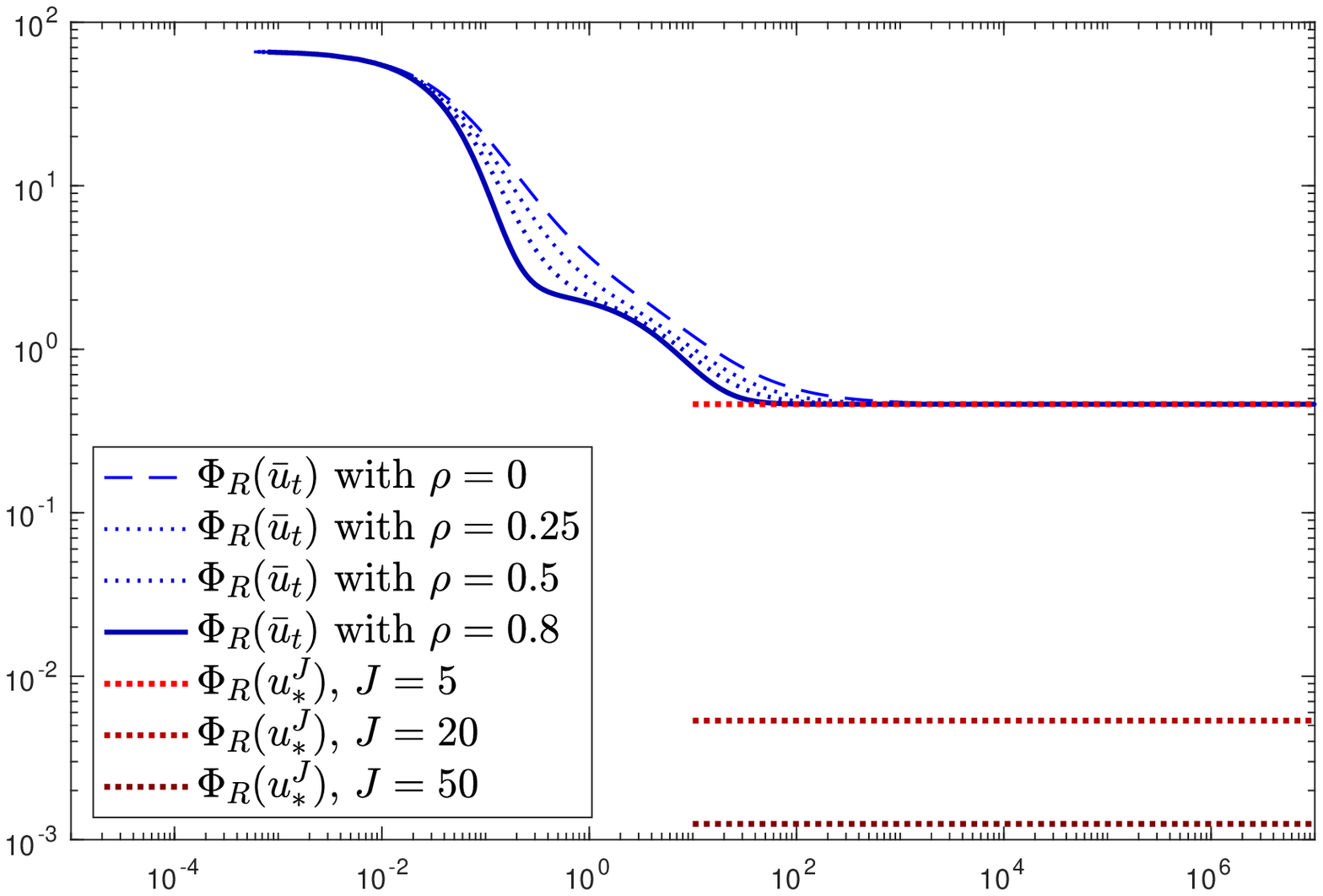}
\includegraphics[scale=0.3]{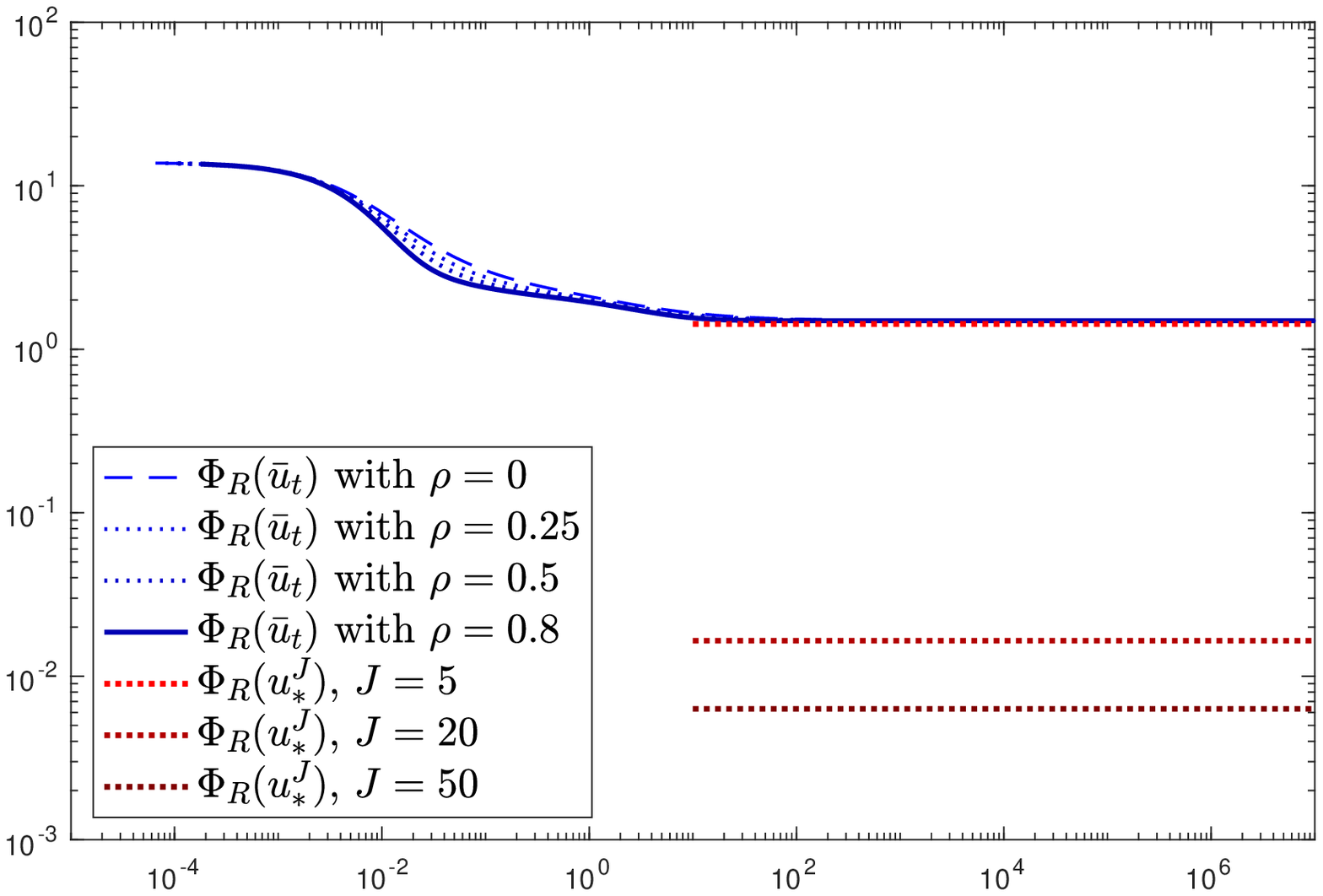}
\caption{Lossfunction $\Phi_R$ evaluated at the particle mean $\bar u_t$ of the TEKI flow for  basis function initialization (left) and random initialization (right) with ensemble size $J=5$.}
\label{fig:Phi_J=5}
\end{figure}

\begin{figure}[htb!]
\centering
\includegraphics[scale=0.3]{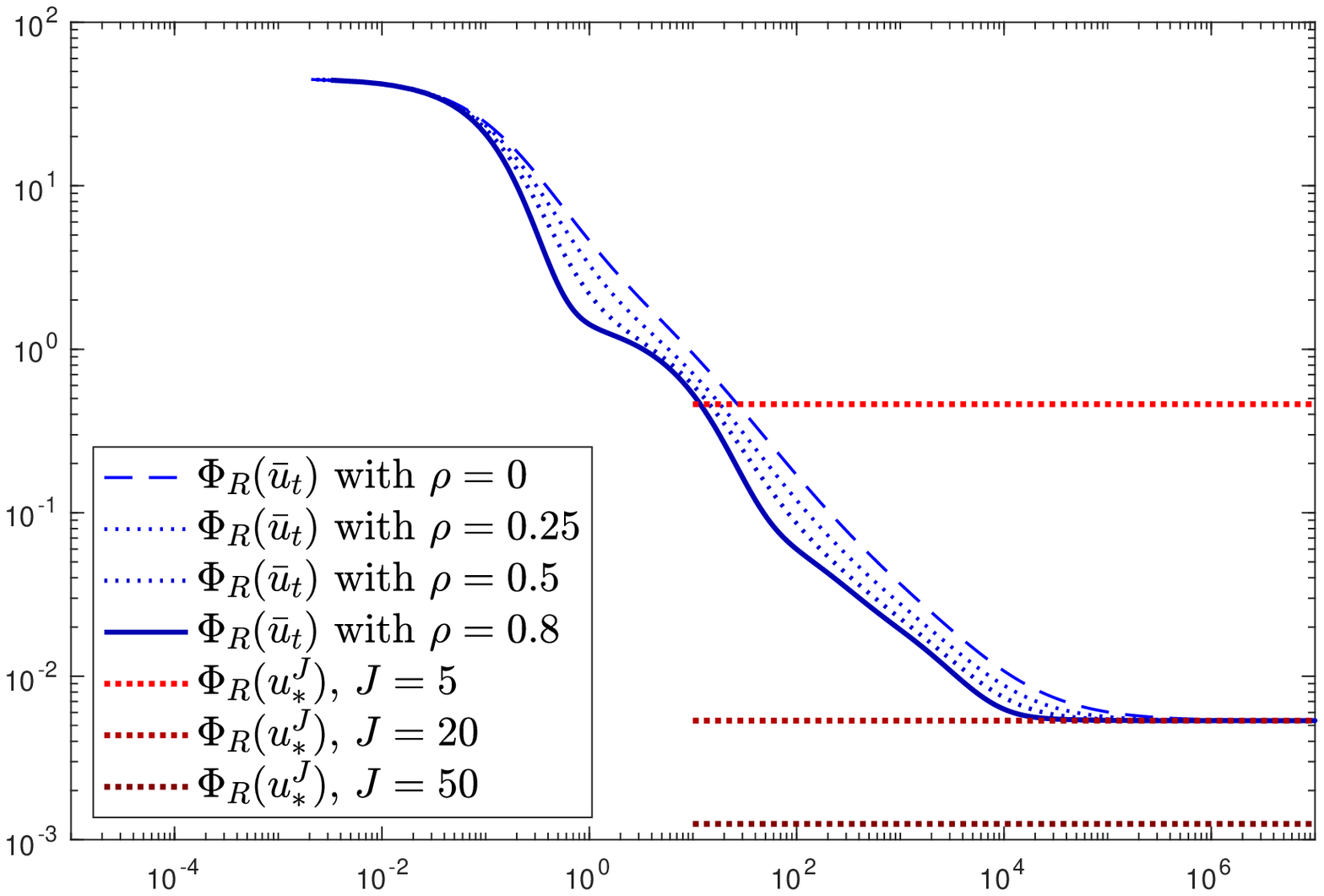}
\includegraphics[scale=0.3]{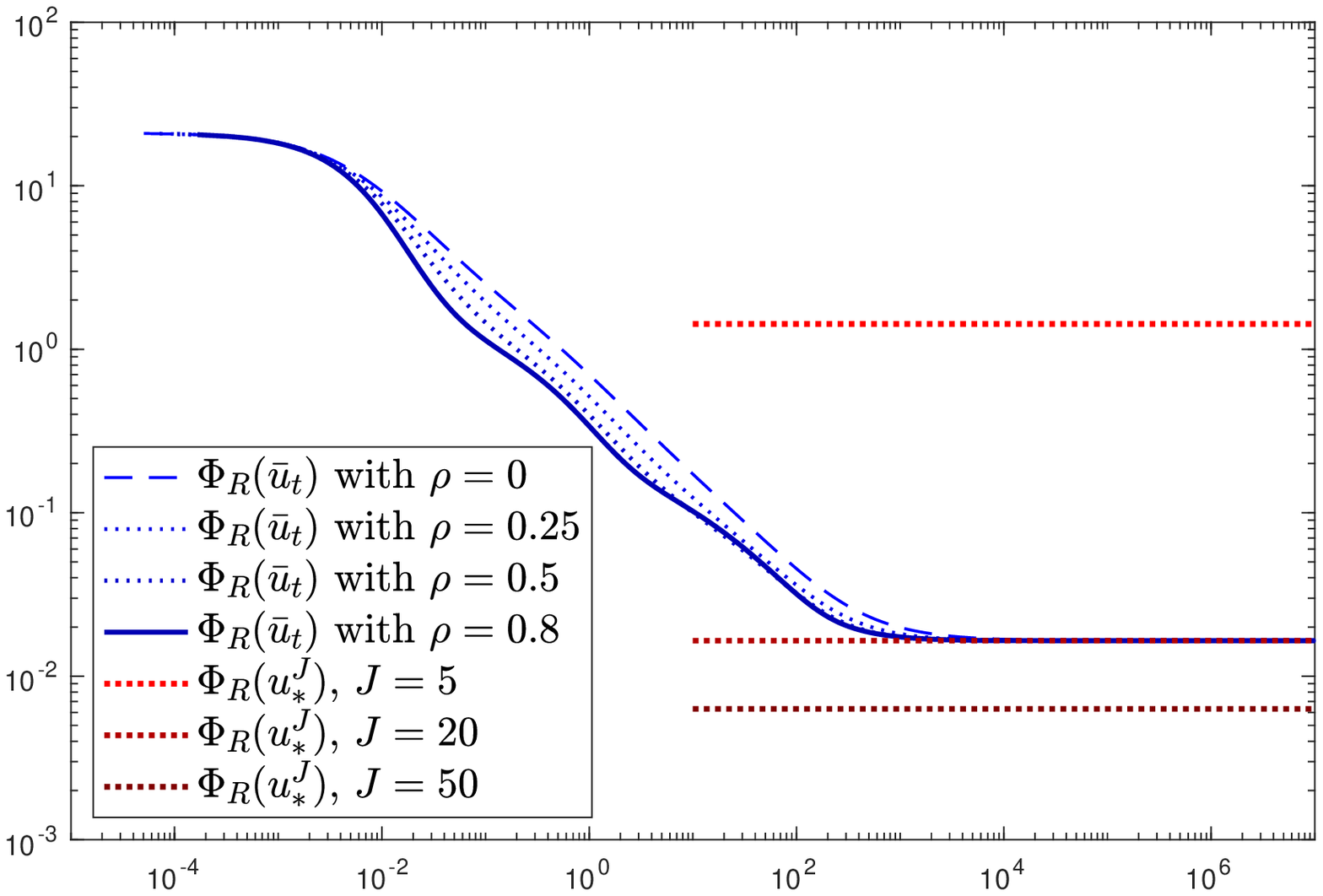}
\caption{Lossfunction $\Phi_R$ evaluated at the particle mean $\bar u_t$ of the TEKI flow for  basis function initialization (left) and random initialization (right) with ensemble size $J=20$.}
\label{fig:Phi_J=20}
\end{figure}

\begin{figure}[htb!]
\centering
\includegraphics[scale=0.3]{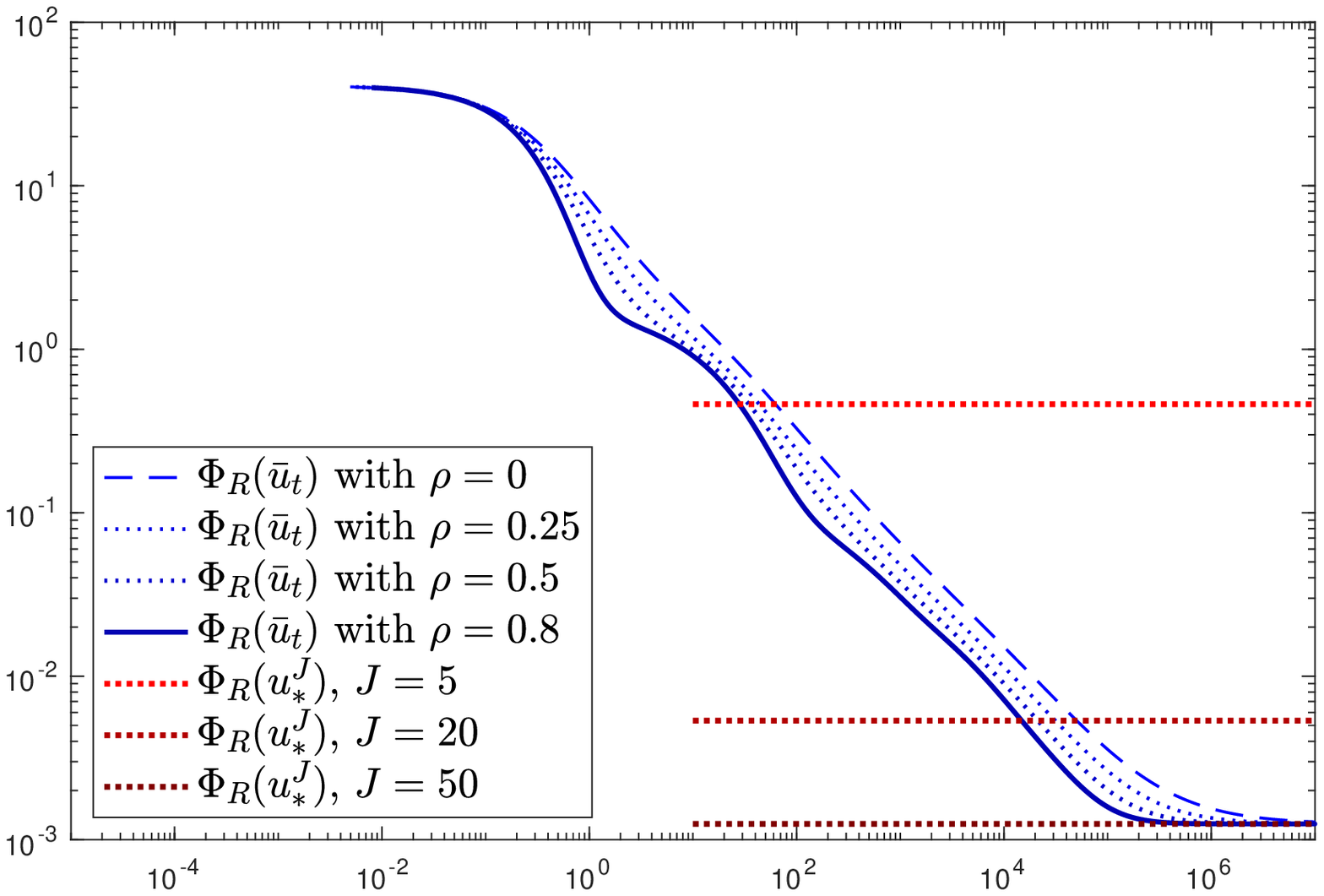}
\includegraphics[scale=0.3]{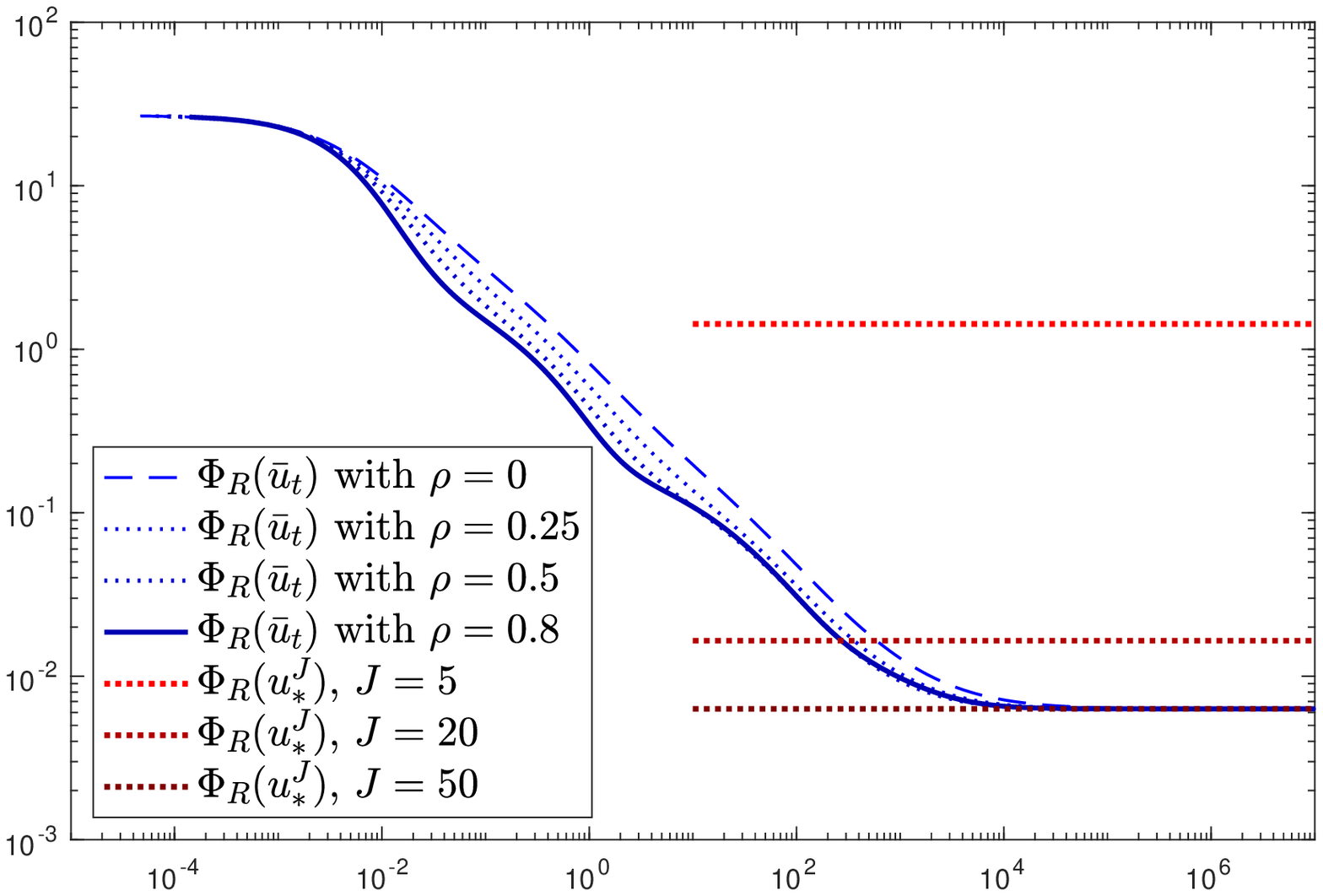}
\caption{Lossfunction $\Phi_R$ evaluated at the particle mean $\bar u_t$ of the TEKI flow for  basis function initialization (left) and random initialization (right) with ensemble size $J=50$.}
\label{fig:Phi_J=50}
\end{figure}

\begin{figure}[htb!]
\centering
\includegraphics[scale=0.3]{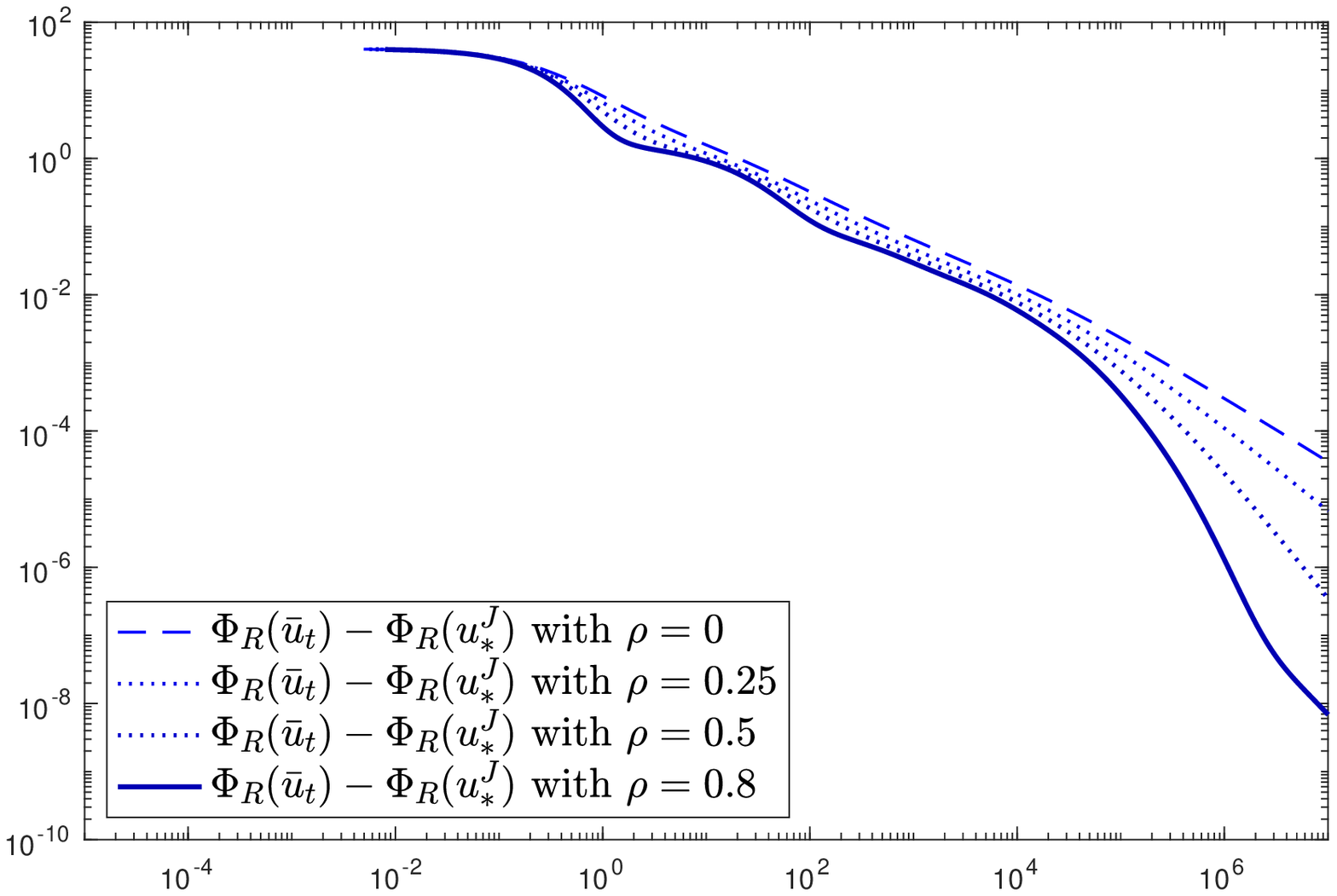}
\includegraphics[scale=0.3]{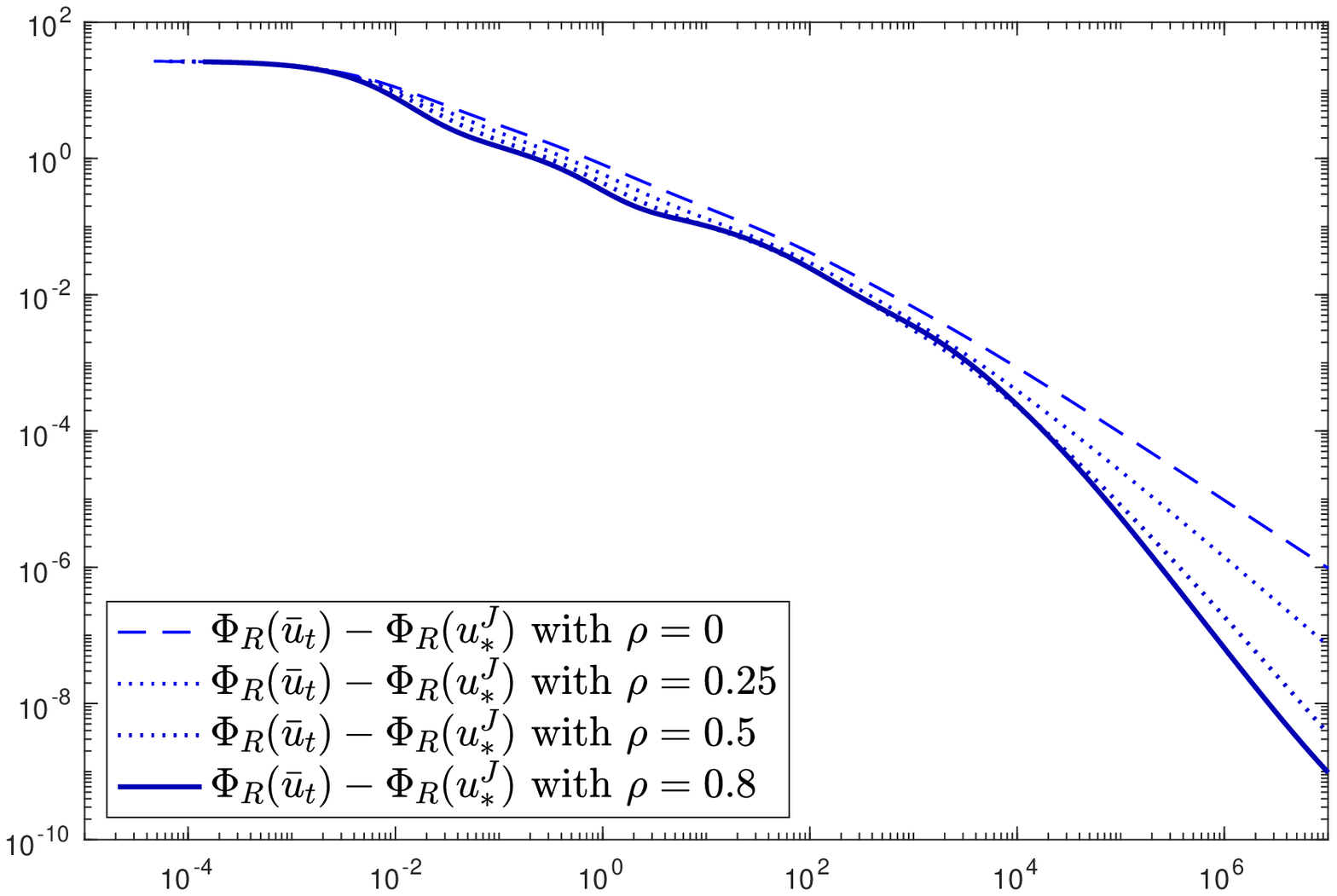}
\caption{Difference of the lossfunction $\Phi_R$ evaluated at the particle mean $\bar u_t$ the TEKI flow for basis function initialization (left) and random initialization (right) with ensemble size $J=50$ and the optimal solution restricted to the initial subspace $u_\ast^J$.}
\label{fig:Phi_diff_J=50}
\end{figure}

\begin{figure}[htb!]
\centering
\includegraphics[scale=0.23]{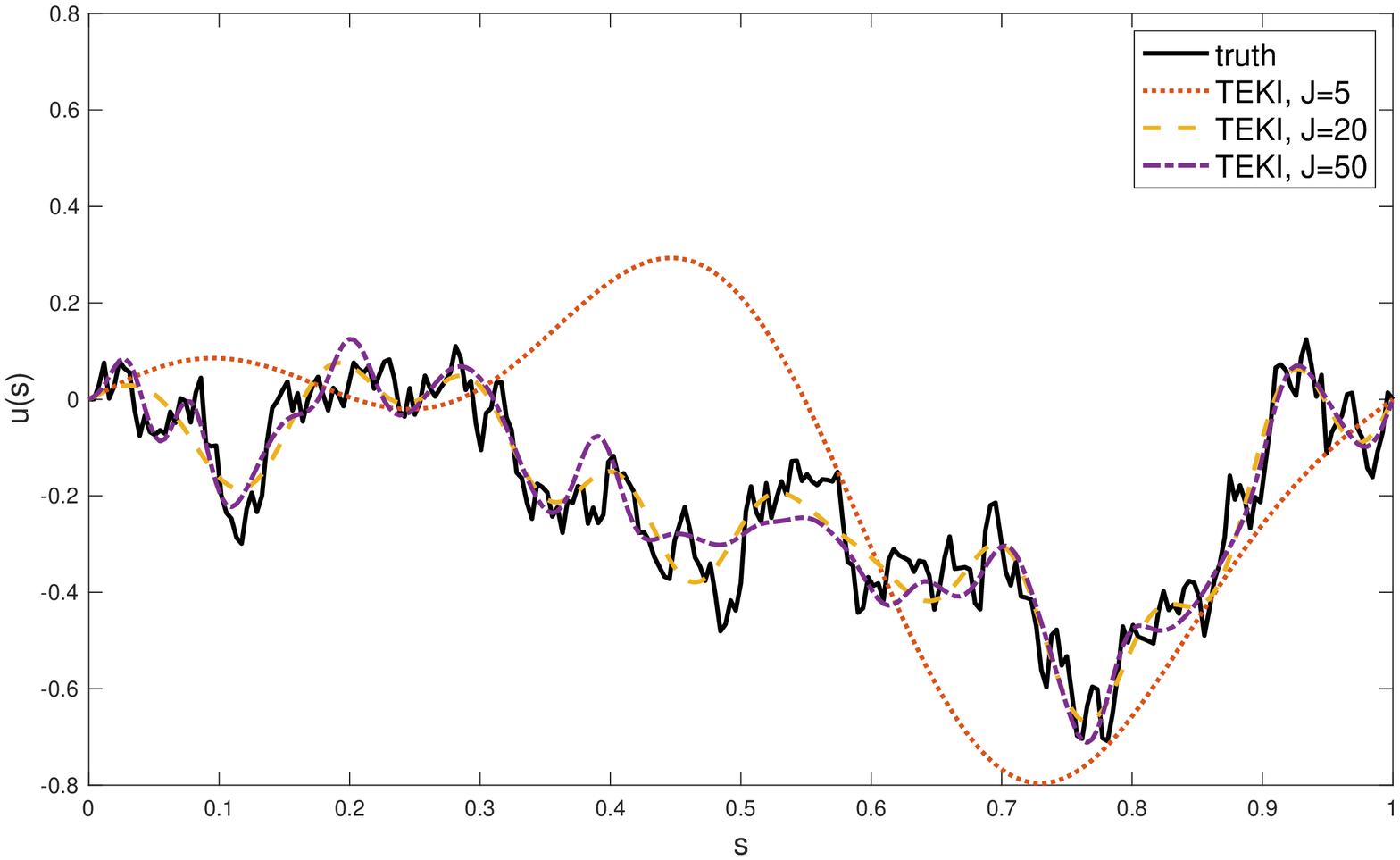}
\includegraphics[scale=0.23]{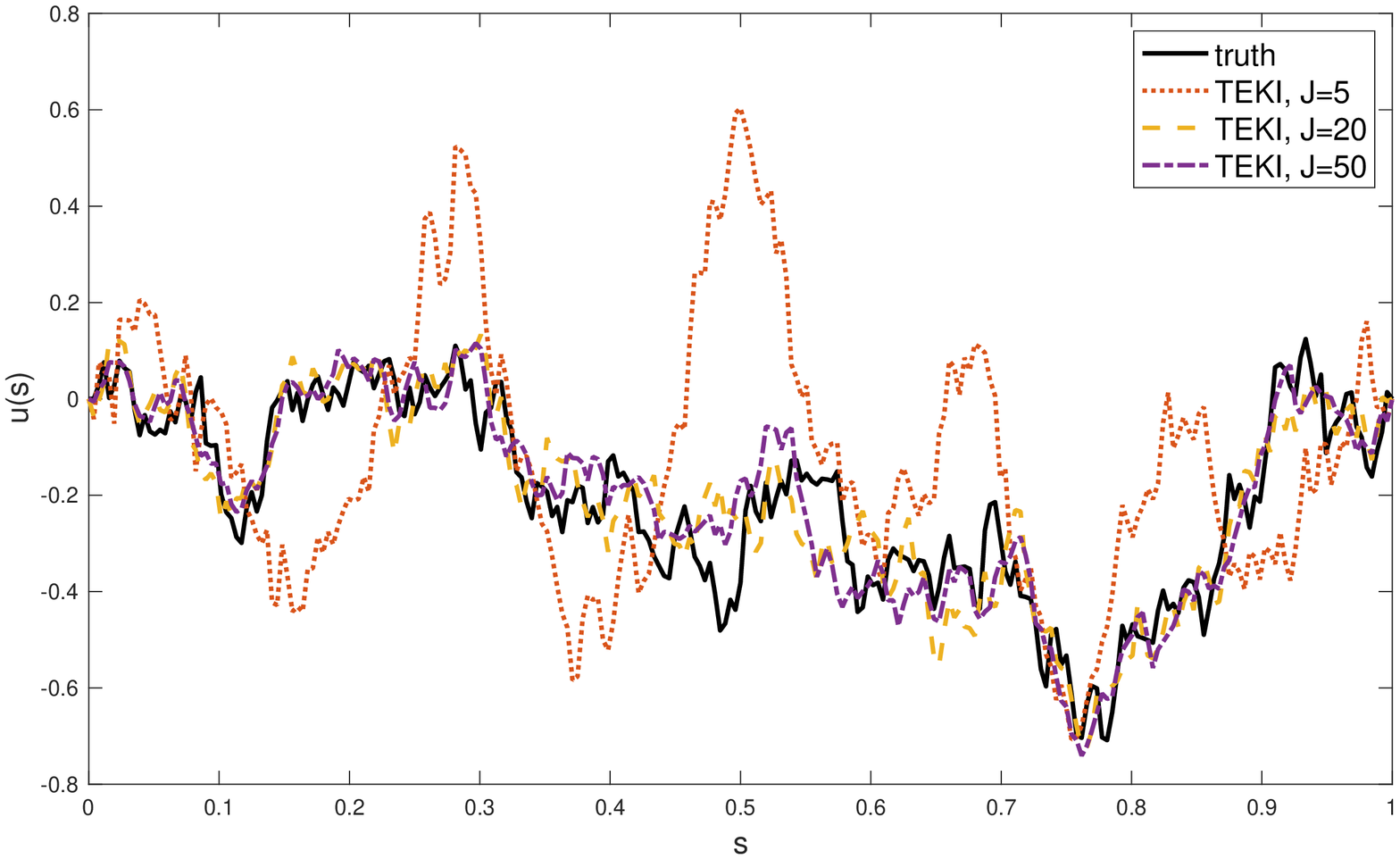}
\caption{Resulting estimates of the unknown parameter $u$ for different ensemble sizes $J\in\{5,20,50\}$ of the TEKI flow initialized by basis functions (left) and random (right). }
\label{fig:par_est}
\end{figure}

\newpage
{
\subsection{Adaptive regularization}
While in the previous experiment we have tested TEKI with fixed regularization parameter in order to verify our theoretical findings, we are providing a second experiment motivated from a more practical point of view. Here, we move away from the optimistic scenario, where the reference $u^\dagger$ is chosen from the correct prior, instead we assume $u^\dagger\sim\mathcal N(0,{C}_0^\dagger)$ with ${C}_0^\dagger:= \tilde{\beta}(-\Delta)^{-\alpha}$,  $\tilde \beta = 1$, $\alpha = 2$. Hence, the incorporated smoothness information through the regularization matrix $C_0$ may be misleading. We therefore propose to apply an adaptive regularization schemes introduced \cite{weissmann2021adaptive} for the stochastic formulation of TEKI.  The idea of choosing a "good" type of Tikhonov regularization is motivated from the hierarchical approach for Bayesian inverse problems. In the classical setting we are interested in the posterior distribution
\[ u\mid y \sim \mu({\mathrm d}x) \propto \exp(-\Phi(x;y))\,\mu_0({\mathrm d}x)\]
whereas in the hierarchical setting we assume that the prior is depending on some unknown hyperparameter $\Theta\in\R^{n_\theta}$ such that we are interested in the posterior
\[(u,\Theta)\mid y \sim \mu({\mathrm d}(x,\theta)) \propto \exp(-\Phi(x;y))\mu_0({\mathrm d}x,\theta)q_0({\mathrm d}\theta) \]
assuming a prior $\Theta\sim q_0({\mathrm d}\theta)$.  We consider the case where $\mu_0$ is given as Gaussian prior with pdf
\[\mu_0(x,\theta) = \frac{1}{\sqrt{2\pi\det(C_0(\theta))}}\exp(-\|x-m_0(\theta)\|_{C_0(\theta)}^2,\]
where the hyperparameter $\theta$ parametrizes the prior mean $m_0(\theta)$ and prior covariance $C_0(\theta)$, and we assume an uniform prior on $\Theta\sim\mathcal U((0,M)^{n_\theta})$. When computing the joint MAP of $(u,\Theta)$ the normalizing constant plays an important role such that the optimization problem reads as
\[\min_{x\in\R^{n_x},\theta\in(0,M)^{n_\theta}}\  \frac12\|G(x)-y\|_\Gamma^2 + \frac12\|u-m_0(\theta)\|_{C_0(\theta)}+\frac12\log(\det(C_0(\theta))).\]
In \cite{weissmann2021adaptive} the authors propose to apply a two-level update scheme, where on the first level one applies TEKI to optimize w.r.t.~$x$ given $\theta$, while in the second level one applies Gradient descent w.r.t.~$\theta$ given $x$.  In continuous time the resulting adaptive TEKI flow then reads as
\begin{align*}\label{eq:twolevelTEKI}
\frac{{\mathrm d}u_t^{(j)}}{{\mathrm d}t} &= C^{u,G}(u_t)\Gamma^{-1}(y-G(u_t^{(j)}))-C(u_t)C_0(\theta_t^{(j)})^{-1}u_t^{(j)},\\
\frac{{\mathrm d}\theta_t^{(j)}}{{\mathrm d}\theta} &= -\nabla_\theta \left(\frac12\|u_t^{(j)}-m_0(\theta)\|_{C_0(\theta^{(j)})} +\frac12\log(\det(C_0(\theta^{(j)})))\right).
\end{align*}
Alternatively, we consider the second level as Gradient flow w.r.t.~$\theta_t$ given the particle mean $\bar u_t$, i.e. 
\begin{equation}\label{eq:twolevelTEKI2}
\begin{split}
\frac{{\mathrm d}u_t^{(j)}}{{\mathrm d}t} &= C^{u,G}(u_t)\Gamma^{-1}(y-G(u_t^{(j)}))-C(u_t)C_0(\theta_t^{(j)})^{-1}u_t^{(j)},\\
\frac{{\mathrm d}\theta_t}{{\mathrm d}\theta} &= -\nabla_\theta \left(\frac12\|\bar u_t-m_0(\theta)\|_{C_0(\theta)} +\frac12\log(\det(C_0(\theta)))\right).
\end{split}
\end{equation}
We note that this type of adaptive choice of regularization remains derivative-free w.r.t.~the forward model $G$. In the following, we consider $m_0(\theta)=0$ and a parametrization of $C_0(\theta)$ of the form
\[C_0(\theta) = V\begin{pmatrix} 1/\theta_1 & & \\
&\ddots&\\
& & 1/\theta_{n_x}
\end{pmatrix}V^\top, \]
where $V\in\R^{n_x\times n_x}$ is an orthonormal matrix built through the eigenfunctions of $C_{{\mathrm{fix}}} = (-\Delta)^{-1}$. We compare the adaptive TEKI flow \eqref{eq:twolevelTEKI2} with the fixed TEKI flow \eqref{eq:TEKI_fixed}  (with $\rho =0$) with $\kappa \in\{0.001, 1\}$ and fixed regularization matrix $C_{{\mathrm{fix}}}$. As mentioned above the underlying ground truth has been drawn from $u^\dagger \sim\mathcal N(0,{C}_0^\dagger)$. We initialize all of the three methods with the same initial ensemble $u_0^{(j)}\sim \mathcal N(0, C_{{\mathrm{fix}}}), \ J=100$ and choose $\theta_0$ such that $C_0(\theta_0) = 1\cdot C_{{\mathrm{fix}}}$.}

\begin{figure}[htb!]
\centering
\includegraphics[scale=0.23]{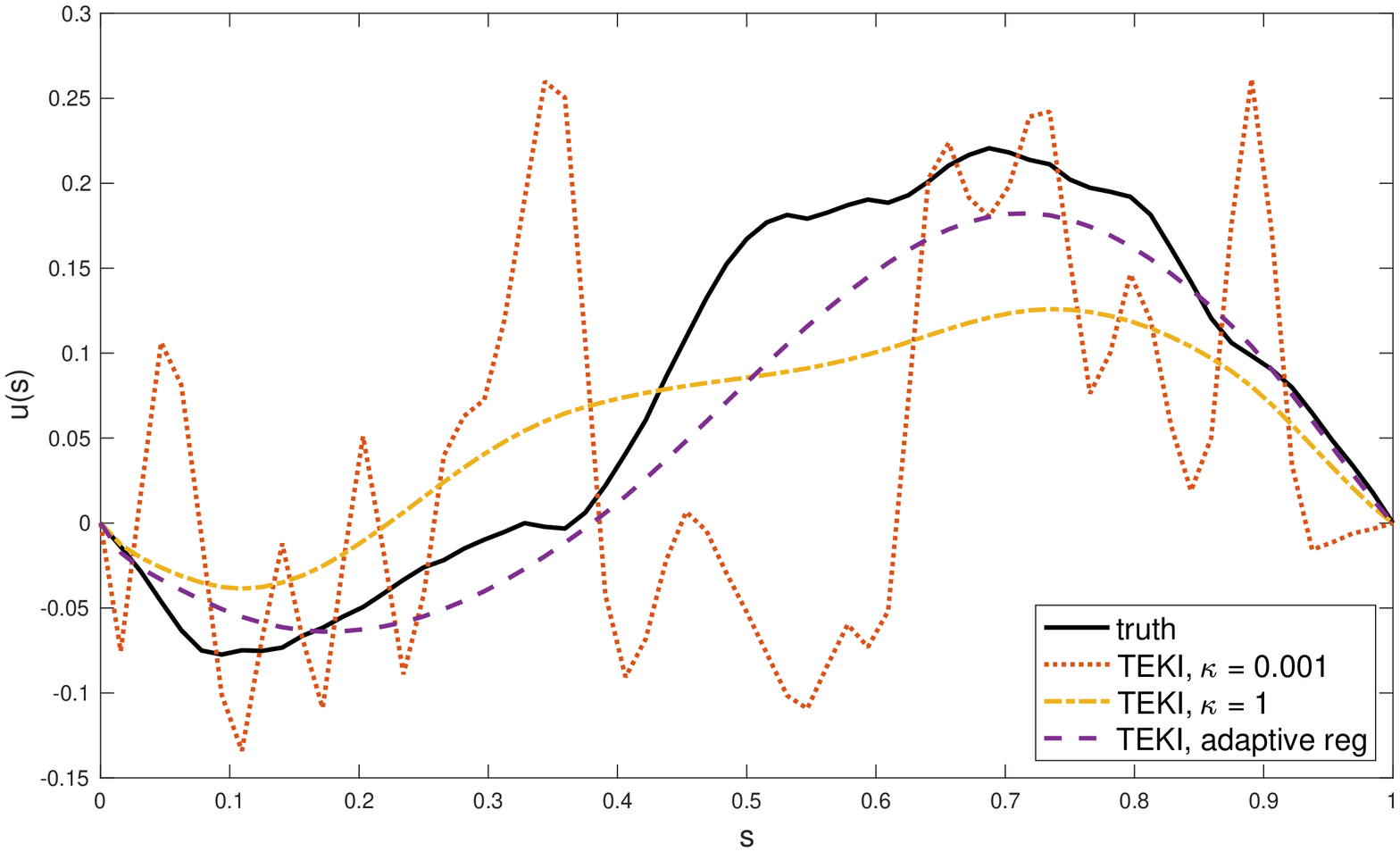}
\includegraphics[scale=0.23]{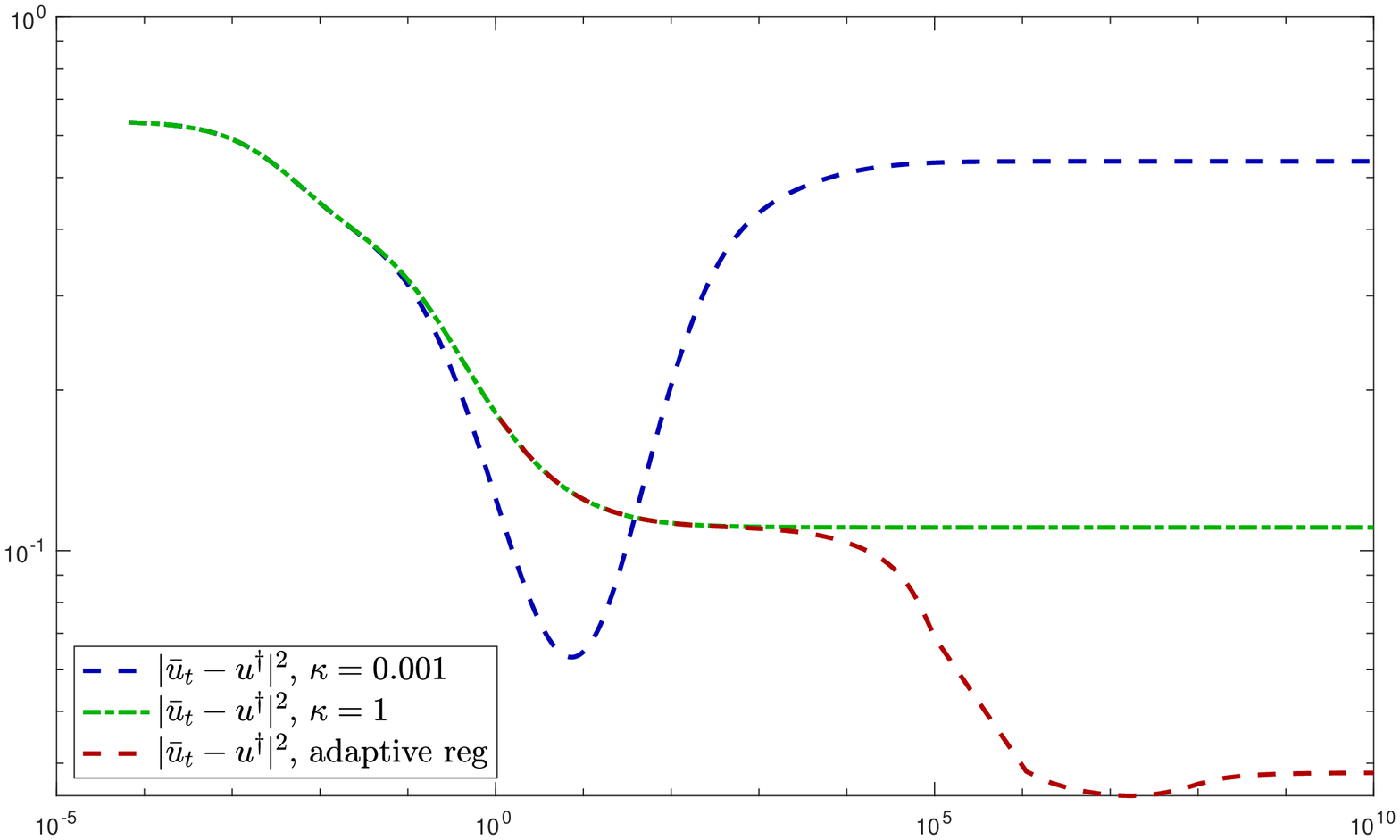}
\caption{{Resulting estimates of the unknown parameter $u^\dagger$ for different types of regularization of the TEKI flow (left) and corresponding residuals over time (right). }}
\label{fig:adaptive_reg}
\end{figure}

{
In Figure~\ref{fig:adaptive_reg} we plot the corresponding parameter estimation for all of the three schemes as well as the residuals $\|\bar u_t - u^\dagger\|^2$ for parameter reconstruction, where $\bar u_t$ denotes the particle mean of the applied (fixed/adaptive) TEKI flow. As expected, for too large choices of $\kappa = 1$ we observe over-smoothing in the parameter reconstruction, whereas for too low choice of $\kappa = 0.001$ we observe over-fitting issues. The adaptive regularization scheme helps to significantly improve the performance of the resulting reconstruction of the underlying ground truth $u^\dagger$.  Finally, we refer to \cite{weissmann2021adaptive} for more details on adaptive regularization schemes for TEKI.}

\section{Conclusion}\label{sec:conclusion}

In this work, we have quantified the EKI as derivative-free optimization method restricted to the finite dimensional subspace spanned through the initial ensemble {spread}.  Our results are based on the deterministic continuous-time formulation represented by an ODE and the incorporation of Tikhonov regularization.  The ensemble collapse has been quantified from above and below such that we were able to bound the approximation of gradients uniform in time and avoiding too fast degeneration of the preconditioner. We have further improved the convergence result through covariance inflation without breaking the subspace property. The considered covariance inflation can be viewed as generalization of the ESRF and gives an intuitive connection between the EKI and ESRF.  In our numerical experiment we have illustrated the application of EKI as optimization method restricted to the initial subspace and specified the role of choosing the initial ensemble.

The next step for future work includes the extension of the presented results to the stochastic formulation of EKI represented as stochastic differential equation. The key challenge in this scenario will be the quantification of the ensemble collapse, since the resulting lower and upper bounds on the eigenvalues will be path-depending.  In addition,  a detailed complexity analysis of EKI needs to be done in order to quantify the advantage of applying EKI as derivative-free scheme compared to other commonly used optimization methods including derivatives such as gradient descent or (Quasi-)Newton method.


\medskip

\noindent
{\bf Acknowledgement.} The author is very grateful for helpful discussions with Claudia Schillings and Jakob Zech which had an significant impact on the contribution of this work. Moreover, the author thanks Neil Chada and Claudia Schillings for carefully proofreading this manuscript.


\bibliographystyle{siam}

\bibliography{references}

\end{document}